\documentclass{llncs} 

\usepackage{amsmath}
\usepackage{amsfonts}
\usepackage{amssymb}
\usepackage{oldgerm}
\usepackage{amscd}

\usepackage{xypic}
\xyoption{all}

\usepackage{epsfig}
\usepackage{graphics}

\usepackage{mathrsfs}

\newcommand{\W}{\mathrm{W}}

\newcommand{\Proj}{\mathbb P}
\newcommand{\N}{\mathbb N}

\title{A point counting algorithm using cohomology with
  compact support} 

\author{Gweltaz Chatel\inst{3} \and David Lubicz\inst{1,2}}

\institute{
  DGA-CELAR, Bruz, France \and IRMAR, Universit\'e de Rennes 1, France \and 
Universit\`a degli Studi di Padova, Italy\\
}

\begin{document}
\maketitle

\begin{abstract}
  We describe an algorithm to count the number of rational points of
  an hyperelliptic curve defined over a
  finite field of odd characteristic which is based upon the
  computation of the action of the Frobenius morphism on a basis of the
  Monsky-Washnitzer cohomology with compact support. This algorithm
  follows the vein of a systematic exploration of potential
  applications of cohomology theories to point counting.

  Our algorithm decomposes in two steps. A first step which consists
  in the computation of a basis of the cohomology and then a second
  step to obtain a representation of the Frobenius morphism.  We
  achieve a $\tilde{O}(g^4 n^{3})$ worst case time complexity and $O(g^3
  n^3)$ memory complexity where $g$ is the genus of the curve
  and $n$ is the absolute degree of its base field. We give a detailed
  complexity analysis of the algorithm as well as a proof of
  correctness.
\end{abstract}

\section{Introduction}
The problem of counting the number of rational points on a smooth
projective algebraic curve defined over a finite field has attracted a
lot of attention in recent years driven by well known
cryptographic applications.

The known point counting algorithms for a curve $C$ over a finite field $k$ with
Jacobian $J(C)$ can roughly be divided in two large classes, the
$\ell$-adic algorithms and the $p$-adic algorithms.

The $\ell$-adic point counting algorithms can be interpreted as the
computation of the action of the Frobenius morphism on the $\ell$-adic
Tate module of the Jacobian, $\ell$ being prime to the characteristic
of the base field. The $\ell$-adic point counting algorithms all
follow an original idea of Schoof~\cite{Schoof1} which consists in
computing the action of the Frobenius morphism on the group of
$\ell$-torsion points of $J(C)$ for primes $\ell$ big enough to
recover the zeta function of the curve by the Chinese remainder
theorem. In the case that $C$ is an elliptic curve, this first
algorithm has subsequently been improved by Elkies, Atkin and other
authors~\cite{MR1486831,Schoof1} and resulted in a very efficient
algorithm. Some of these techniques have been adapted in the case of
higher genus curves~\cite{Pila1990} but some improvement have yet to
be done in order to be able to reach cryptographic size fields.
Nevertheless some progress has been made~\cite{Gaudry,Gaudrynt} in
that direction.

On the other side, the $p$-adic point counting algorithms can be
interpreted as the computation of the action of the Frobenius morphism
on some $p$-adic cohomology group. The first such algorithm has been
described by Satoh~\cite{MR2001j:11049}. The algorithm of Satoh relies
on the computation of the action of the Frobenius morphism on the
invariant differential forms of the canonical lift of an ordinary
Jacobian. It has been generalised and made more efficient in a series
of papers \cite{SaSkTa2001,Mestre4,Mestre5,Gaudry2002,LL03,MR2293798}.
Other authors have explored other possible representations of the
Frobenius morphism. In~\cite{Kedlaya2001}, Kedlaya explains how to
obtain a basis of the Monsky-Washnitzer cohomology of an hyperelliptic
curve and compute the Frobenius morphism acting on it. Lauder used
Dwork cohomology to obtain general point counting algorithms
\cite{MR1916921,MR2067435}. Later on, Lauder introduced deformation
techniques in which one consider a one parameter family of curves and
use the Gauss-Manin connection in order to carry over this family the
action of the Frobenius morphism~\cite{MR2044050}.

Pursuing the exploration of possible $p$-adic cohomology theories, we
propose in this paper to use Monsky-Washnitzer cohomology with compact
support.  This cohomology theory comes with a Lefschetz trace formula
and as a consequence can be used to compute the number of rational
points of a curve defined over a finite field. Starting with an
hyperelliptic curve defined over a finite field of odd characteristic,
we explain how to represent elements of the Monsky-Washnitzer
cohomology with compact support and obtain a basis of this vector
space.  Then we compute the representation of the Frobenius morphism.

Our algorithm breaks into two steps. The first step is the computation
of a basis of the Monsky-Washnitzer cohomology with compact support.
Unlike the case of the algorithm of Kedlaya this step is non trivial
from an algorithmic point of view. The two main features of this part
of our algorithm is:
\begin{itemize}
\item The use of the stability of the cohomology of Monsky-Washnitzer
  with compact support by finite \'etale descent to reduce the
  computation of a basis of the cohomology to the computation of a
  basis of the global horizontal sections of an isocrystal over the affine
  line. These global sections verify a differential equation provided
  by the Gauss-Manin connection that we interpret as the solutions of
  a linear system. This is the so called global method. 
\item Unfortunately, the global method is inefficient since it
  involves the inversion of a matrix the size of which is in the order
  of the analytic precision required for the computations. We explain
  how to speed up the computation of the basis by taking advantage of
  local inhomogeneous differential equations deduced from the Gauss-Manin
  connection. Once we have computed global solutions for the
  Gauss-Manin connection up to a small analytic precision it is
  possible to prolong them locally using an asymptotically fast algorithm
  such as~\cite{BostanCOSSS07}.
\end{itemize}

The second step of the algorithm is the computation of a lift of the
Frobenius morphism as well as its action on a basis of the cohomology.
The computation of a lift of the Frobenius morphism that we describe
in this paper is standard, being just an adaptation of
\cite{Kedlaya2001} to our case. Nonetheless, for the action of the
Frobenius morphism, we explain how it is possible to turn into an
efficient algorithm the knowledge of local differential equations
deduced from the twist of the Gauss-Manin connection by the Frobenius
morphism.

We give a proof of correctness for our algorithm. As usual for the
$p$-adic algorithms, the main problem lies in the assessment of the
analytic and $p$-adic precisions necessary to recover the
characteristic polynomial of the Frobenius morphism. For this we used
a variant of a result of Lauder~\cite{MR2261044}. We provide the
algorithm with a detailed complexity analysis.  In this complexity
estimate, we suppose that the characteristic $p$ of the base field of
the hyperelliptic curve is fixed and consider complexity bounds when
the genus $g$ and the absolute degree $n$ of the base field increase.
Using the soft-O notation to neglect the logarithmic terms, we obtain
a $\tilde{O}(g^4 n^{3})$ time complexity with a $O(g^3 n^3)$ memory
consumption. In order to achieve this time complexity, we use in an
essential manner asymptotically fast algorithms to compute with power
series \cite{MR0398151,BoSaSc08,MR0314238,BostanCOSSS07}.  We remark
that the worst case complexity of our algorithm has the same
complexity bound in $n$ as the algorithm of Kedlaya but nonetheless a
better behaviour with respect to the genus.

\paragraph{Organisation of the paper} In Section \ref{secmath1}, we
recall the basic construction of the cohomology of an overconvergent
isocrystal over the affine line. In Section \ref{secalgo1}, we deduce
from it an algorithm to compute a basis of the cohomology. We give a
proof of correctness and a detailed complexity analysis of this
algorithm in Section \ref{secalgo2}. In Section \ref{seclift}, we
explain how to compute the action of the Frobenius morphism on this
basis and recover the zeta function of the curve. In Section
\ref{secconl}, we obtain complexity estimates for the whole algorithm
and discuss practical results and implementations.

\paragraph{Notations} In all this paper, $k$
is a finite field of odd characteristic $p$. We denote by $W(k)$ the
ring of Witt vectors with coefficients in $k$ and by $K$ the field of
fractions of $W(k)$. For all $x \in K$, $|x|$ denotes the usual
$p$-adic norm of $x$.

We use the notation $\underline{\ell}$ for a multi-indice $(\ell_0, \ldots,
\ell_{k})$ where $k>0$ is an integer. For instance, $\underline{0}$ is
the $k$-fold multi-indice $(0, \ldots, 0)$. For $\underline{\ell}$ and
$\underline{m}$ two multi-indices $\underline{\ell} > \underline{m}$
means that for all $i \in \{ 0, \ldots , k \}$, $\ell_i > m_i$. Moreover
if $\underline{\ell}$ is a multi-indice, we set $|\underline{\ell}|= \max \{
|\ell_i|, i =0, \ldots , k \}$.

We recall that an element $f=\sum_{\underline{\ell}} a_{\underline{\ell}}
t^{\underline{\ell}}$ of $K[[t_0, \ldots , t_k]]$ is called overconvergent
if there exists $\eta_0 >1$ such that $\lim |a_{\underline{\ell}}|
\eta_0^{|\underline{\ell}|}=0$. The sub-ring of $K[[t_0, \ldots ,t_k]]$
consisting of overconvergent elements is denoted by $K[t_0, \ldots
,t_k]^\dagger$ and is called the weak completion of $K[t_0, \ldots ,
t_k]$.

\section{Basic definitions}\label{secmath1}
In order to fix the notations, we first recall some basic facts about
Monsky-Washnitzer cohomology with compact support. In the same way as
the theory without support, the Monsky-Washnitzer cohomology with
compact support associates to a smooth affine curve $C_k$ of genus $g$
over the finite field $k$ a graded $K$-vector space of finite
dimension denoted by $(H^i_{MW,c}(C_k/K))_{i \in \{0,1,2\}}$. There
exists a trace formula in this theory which can be used to recover the
zeta function of $C_k$ by computing the action of the Frobenius
morphism over the cohomology groups. We remark that as we are dealing
with curves, only $H^1_{MW,c}(C_k/K)$ is non trivial.  Moreover, we
will use the property of stability of this cohomology theory by finite
\'etale descent to do all our computations over the affine line. This
entails working with non-trivial coefficients which are overconvergent
modules with connection.

\subsection{The geometric setting}
In the following, we focus on the case of hyperelliptic curves which
constitute the simplest family of curves with arbitrary genus.  Let
$k$ be a finite field of odd characteristic.  Let $C_k$ be the affine
model of a smooth genus $g$ hyperelliptic curve over $k$ given by an
equation $Y^2=\prod_{i=1}^{2g+1} (X - \overline{\lambda}_i)$ where
$\overline{\lambda}_i \in k$.  Let $\pi_k: C_k \rightarrow
\Bbb{A}^1_k$ be the projection along the $Y$-axis.  If we denote by
$U_k$ the \'etale locus of $\pi_k$ and by $V_k$ its image, we have a
diagram
\begin{equation}\label{diag1}
\xymatrix{
  C_k \ar[d]^{\pi_k}  & U_k \ar[d]^{\pi_k} \ar@{_{(}->}[l]\\
  \Bbb{A}^1_{k} & V_k \ar@{_{(}->}[l] },
\end{equation}
where the horizontal maps are open immersions and where $\pi_k$ is
finite \'etale over $V_k$. We let $A_k$ and $B_k$ be the coordinate
rings of $U_k$ and $V_k$. By~\cite{MR0345966}, Theorem 6, we can lift
diagram (\ref{diag1}) to a diagram
\begin{equation}
\xymatrix{C \ar[d]^{\pi}  & U \ar[d]^{\pi} \ar@{_{(}->}[l]\\
  \Bbb{A}^1_{W(k)} & V \ar@{_{(}->}[l] },
\end{equation} 
of smooth $W(k)$-schemes where the horizontal maps are open immersions
and $\pi$ is finite \'etale over $V$. Let $A$ and $B$ be the
coordinate rings of $U$ and $V$ respectively. Let $\Lambda_k=\{
\overline{\lambda}_1,\ldots,\overline{\lambda}_{2g+1},\overline{\infty}
\}$ be the complement of $V_k(\overline{k})$ in
$\Bbb{P}^1_k(\overline{k})$. We can suppose that $B=W(k)[t,
(t-\lambda_1)^{-1}, \ldots, (t-\lambda_{2g+1})^{-1}]$ where $t$ is an
indeterminate and where $\lambda_i \in W(k)$ lifts
$\overline{\lambda}_i$ for $i=1, \ldots , 2g+1$. Let $\Lambda = \{
\lambda_1,\ldots,\lambda_{2g+1},\infty \}$, $A_K = A \otimes_{W(k)} K$
and $B_K = B\otimes_{W(k)} K$.  By finite \'etale descent (cf
Corollary 2.6.6 of \cite{MR1728610}) we have $H^1_{MW,c}(U_k/K)=
H^1_{MW,c}(V, \pi_* A^\dagger_K)$.  In the following, we always
consider $A_K^\dagger$ with its $B^\dagger_K$-module structure
provided by $\pi$.

There is a Gauss-Manin connection on $A_K^\dagger$ and the computation
of a basis of the cohomology of a curve boils down to the computation
of horizontal sections for this connection. In the remaining, if $t$
is an affine parameter, we will use the notation $(t-\infty)=t^{-1}$
to indicate a local parameter at the infinity.

\subsection{The space $B_c$}\label{A_c}

In this section, we give a definition of the Monsky-Washnitzer
cohomology with compact support over a Zariski open subset of the
affine line. Note that in this situation this theory coincides with
the so-called dual theory of Dwork~\cite{MR1085482}. Let $B$,
$\Lambda$ as before.

\begin{definition}
For $\lambda \in \Bbb{P}^1_K$ rational, let 
\begin{equation*}
\begin{split}
R_\lambda & =\big\{\sum_{n\in \Bbb{Z}} a_\ell (t-\lambda)^\ell|a_\ell \in K, \forall
\eta <1, |a_\ell|\eta ^{\ell} \rightarrow _{+\infty} 0\text{, and } \\ & \exists
\eta_0 >1, |a_\ell|\eta_0^{|\ell|} \rightarrow _{-\infty} 0 \big\},
\end{split}
\end{equation*}
be the ring of power series converging on a subset
$\eta_0^{-1}<|t-\lambda|<1$ with $\eta_0>1$. The Robba ring associated
to $B$ is the ring $$\mathcal{R}_B=\oplus_{\Lambda}R_\lambda.$$
\end{definition}

We recall that the weak completion of $B_K=K[t, (t-\lambda_1)^{-1}, \ldots ,
(t-\lambda_{2g+1})^{-1}]$ is
\begin{equation*}
\begin{split}
  B^\dagger_K & =\big\{\sum_{\underline{\ell}\geq \underline{0}}
  a_{\underline{\ell}}
  t^{\ell_0}(t-\lambda_1)^{-\ell_1}\ldots(t-\lambda_{2g+1})^{-\ell_{2g+1}}|
  a_{\underline{\ell}} \in K, \exists \eta_0 >1, \\
  &|a_{\underline{\ell}} |\eta_0^{|\underline{\ell} |} \rightarrow
  _{|\underline{\ell}| \to \infty} 0 \big\}.
\end{split}
\end{equation*}

The space $B_c$ of analytic functions with compact support is by
definition the quotient of $\mathcal{R}_B$ by the image of a certain
map from $B^\dagger_K$ into $\mathcal{R}_B$ that we describe in the
following.  For $\lambda \in \Lambda$, let
\begin{equation}\label{philambda}
\phi_\lambda: B_K^\dagger \rightarrow R_\lambda
\end{equation}
be the injective map which sends $b_K \in B_K^{\dagger}$ to the
element $b_\lambda \in R_\lambda$ which is the local development of
$b_K$ at $\lambda$.

Let $i_D: B_K^\dagger \rightarrow \mathcal{R}_B$, be defined as $b_K
\mapsto \bigoplus_{\lambda \in \Lambda} \phi_\lambda(b_k)$. Obviously,
$i_D$ is an injection and by definition the space $B_c$ is the
quotient of $\mathcal{R}_B$ by the image of $i_D$. As a consequence,
we obtain the short exact sequence of $K$-vector spaces
\begin{equation}\label{suiex}
\xymatrix{
0 \ar[r] & B^\dagger_K \ar[r]^{i_D} & \mathcal{R}_B \ar[r]^{p_c} & B_c
\ar[r] & 0 \\
},
\end{equation}
where $p_c$ is the canonical projection.

Actually, the space $B_c$ comes with a natural structure of $B^\dagger_K$-module
that we describe now. To define it, we let
$$R^\dagger_\lambda=\big\{ \sum_{\ell < 0} a_n (t-\lambda)^\ell|a_\ell \in
K, \exists \eta_0 >1, |a_{-\ell}|\eta_0^{\ell} \rightarrow _{+\infty} 0 \big\},$$
and
$$\tilde{R}^\dagger_\lambda=\big\{ \sum_{\ell \leq 0} a_\ell (t-\lambda)^\ell|a_\ell \in
K, \exists \eta_0 >1, |a_{-\ell}|\eta_0^{\ell} \rightarrow _{+\infty}
0 \big\}.$$

And we put the
\begin{definition}\label{principal}
  Let $\Lambda_0=\Lambda \backslash \{ \infty \}$. The principal part at
  $\lambda \in \Lambda_0$ is the function defined by
\begin{center}
$\Pr_\lambda: R_\lambda \rightarrow R^\dagger_\lambda,$
\end{center}
$${\Pr}_\lambda(\sum_{\ell\in \Bbb{Z}} a_\ell (t-\lambda)^\ell)=\sum_{\ell < 0}
a_\ell (t-\lambda)^\ell.$$ 
The principal part at $\infty$ is the function defined by
$${\Pr}_\infty: R_\infty \rightarrow \tilde{R}^\dagger_\infty,$$
$${\Pr}_\infty(\sum_{\ell\in \Bbb{Z}} a_\ell t^{-\ell})=  \sum_{\ell
  \leq 0} a_\ell t^{-\ell}.$$
We also define the analytic part at any $\lambda \in \Lambda$ as the
identity minus the principal part at $\lambda$.
\end{definition}

For all $b_c \in B_c$, by the Mittag-Lefler theorem, there exists a
unique element $\sigma(b_c)$ of $\mathcal{R}_B$, such that
$p_c(\sigma(b_c))=b_c$ and for all $\lambda \in \Lambda$,
$Pr_\lambda(\sigma(b_c))=0$. In this way, we have defined a map
$\sigma$ from $B_c$ to $\mathcal{R}_B$ and it is immediately verified
that $\sigma$ is a section of $p_c$ in the exact sequence (\ref{suiex}).
We now identify $B_c$ with its image by $\sigma$ so that we can write
$$B_c = \oplus _{\lambda \neq \infty} \tilde{R}_{\lambda,c} \oplus R_{\infty,c}$$

where we denote for $\lambda \in \Bbb{P}^1_K$ rational
$$R_{\lambda,c}=\{\sum_{\ell> 0} a_\ell (t-\lambda)^\ell|a_\ell \in K, 
\forall \eta <1, |a_\ell|\eta ^{\ell} \rightarrow _{+\infty} 0\},$$
and
$$\tilde{R}_{\lambda,c}=\{\sum_{\ell\geq 0} a_\ell (t-\lambda)^\ell|a_\ell \in K, 
\forall \eta <1, |a_\ell|\eta ^{\ell} \rightarrow _{+\infty} 0\}.$$ It
is important for the following to remark that $\sigma \circ p_c$ is
given locally as the identity minus the local expansion of the sum of
all the principal parts.

The action of $B^\dagger_K$ over $B_c$ is given by
$$f.g=p_c(i_D(f).\sigma(g))$$
where $f\in B^\dagger_K$, $g \in B_c$ and $.$ is the product in
$\mathcal{R}_B$.

\subsection{The space $M_c$}
The computation of the space $H^1_{MW,c}(U_k/K)$ boils down to the
computation of the de Rham module of analytic forms with compact
support.  Applying the finite \'etale descent
theorem~\cite[Cor.2.6.6]{MR1728610}, this space of analytic functions
is
$$M_c=A_K^\dagger \otimes_{B^\dagger_K} B_c.$$
For $\lambda \in \Lambda_0$, let $M_{c,\lambda}=A_K^\dagger
\otimes_{B^\dagger_K} \tilde{R}_{\lambda,c}$ and let $M_{c,\infty}=A_K^\dagger
\otimes_{B^\dagger_K} R_{\infty,c}$. We have,
$$M_c=\bigoplus_{\lambda \in \Lambda} M_{c,\lambda}.$$ 
An element of $M_{c,\lambda}$ can be written as
$$m_{\lambda}=\sum_{j=0,1} Y^j \sum_{\ell=0}^{\infty}
b^{\lambda}_{j,\ell} (t - \lambda)^\ell.$$
with $b^{\lambda}_{j,\ell}\in K$ and with $b^{\infty}_{j,0}=0$. We keep
the convention of notation $(t-\infty)=t^{-1}$. 
The finite $B^\dagger_K$-module $M_c$ comes with a connection
given by 
$$\nabla_c:M_c \rightarrow M_c \otimes
_{B^\dagger_K} \Omega^1_{B^\dagger_K},$$
$$ m \otimes g_c \mapsto \nabla_{GM} (m) \otimes g_c + m \otimes
\frac{\partial}{\partial t}g_c.dt,$$ where $\nabla_{GM}$ is the
natural Gauss-Manin connection on $A^\dagger_K$. 
In our case this natural connection is given by the partial derivative
with respect to $Y$ acting on the $B^\dagger_K$-module $A^\dagger_K$.
By definition, the space $H^1_{MW,c}(V, \pi_* A^\dagger_K)$ is the
kernel of $\nabla_c$.  As a consequence to compute a basis of
$H^1_{MW,c}(V, \pi_* A^\dagger_K)$ we have to compute a basis of the
space of solutions of the differential equation
\begin{equation}\label{eqcon}
\nabla_c(m_c)=0
\end{equation}
defined over $M_c$. Note that by classical results (see for example
\cite{MR1354350}) the dimension of this space is equal to $2g$ plus
the number of points we took off the affine line, which in our case
gives $4g+1$.

\section{An algorithm to compute a basis of an overconvergent
  isocrystal}\label{secalgo1}

We show in this section that the solutions of Equation (\ref{eqcon})
can be computed by solving a linear system over $K$. First, we explain
how the action of a linear endomorphism of $M_c$ with rational
coefficients can be computed up to a certain analytic precision by
solving a system of linear equations. From this, we deduce two
methods, the global method given in Section \ref{globm} and the local
method presented in Section \ref{loc}, for the computation of a basis
of solutions of Equation (\ref{eqcon}). The global method is slow but
useful to compute the first analytic development of the solutions
required for the quicker local method.

\subsection{Action of a rational endomorphism of $M_c$}\label{glob}
Let $Mat$ be a square matrix of dimension $2$ with coefficients in the
field of rational functions in the indeterminate $t$ over $K$. If we
set $m_c=m\otimes g_c \in M_c=M\otimes B_c$ and using the basis $\{1,Y
\}$ of $M_c$ over $B_c$ to write $m_c$ as a column vector of dimension
$2$ with coefficients in $B_c$, our aim is to compute
$$Mat.m_c=(Mat.m)\otimes g_c.$$
We rewrite $Mat$ as the quotient of a matrix with polynomial
coefficients by a polynomial with the lowest possible degree denoted
by $\Delta$. We suppose that the roots of $\Delta$ are contained in
$\Lambda$.  Let $o_\lambda$ be the order of the roots of $\Delta$ at
the point $\lambda \in \Lambda_0$. Let $m_o= \max_{\lambda \in
  \Lambda_0}(o_\lambda)$.  For the rest of the section, we denote by
$M^\vee$ the transpose of a matrix $M$.

We explain how we associate vectors with coefficients in $K$ to
elements of $M_c$.
\begin{definition}\label{vect}
Let $n>0$ be a positive integer. Let $m_c \in M_c$ that we can write as 
$(m_{\lambda_1},\ldots,m_{\lambda_{2g+1}},m_{\infty})$ with
$$m_{\lambda_i}=\sum_{j=0,1} Y^j \otimes \sum_{\ell=0}^{\infty}
b^{\lambda_i}_{j,\ell} (t - \lambda_i)^\ell,$$ where
$b^{\lambda_i}_{j,\ell}\in W(k)$ and $b^{\infty}_{j,0}=0$, following
our conventions.  For all $\lambda \in \Lambda$, we let
\begin{equation}
v^\lambda_{m_c,n}=(b^{\lambda}_{0,0},b^{\lambda}_{0,1},\ldots
,b^{\lambda}_{0,n},b^{\lambda}_{1,0},\ldots,b^{\lambda}_{1,n})^\vee,
\end{equation}
and denote by $v_{m_c,n}$ the vector
\begin{equation}
v_{m_c,n}=(b^{\lambda_1}_{0,0},b^{\lambda_1}_{0,1},\ldots
,b^{\lambda_1}_{0,m_o+n},\ldots,b^{\lambda_1}_{1,m_o+n},  
\ldots,b^{\infty}_{1,m_o+n})^\vee,
\end{equation}
which can be written by blocks as
\begin{equation}\label{vector}
v_{m_c,n}=(v^{\lambda_1,\vee}_{m_c,m_0+n},\ldots,v^{\infty,\vee}_{m_c,m_0+n})^\vee.
\end{equation}
\end{definition}

\begin{definition}\label{mat+}
Let $n \geq 0$ be an integer. Let $h$ be a rational function in the
indeterminate $t$ with coefficients in $K$. Let $S_{h, \lambda}$ be
the Laurent series obtained by expending $h$ around $\lambda \in
\Lambda$. We write
$$S_{h,\lambda}=a_o(t-\lambda)^{-m_o}+a_1(t-\lambda)^{-m_o+1}+\ldots+a_{m_o+n}(t-\lambda)^{n}+\ldots$$ 
and define $M_{h,\lambda,n}^{+}$ a matrix of size $(n+1,m_o+n+1)$, by

$$\begin{array}{cccc}
M_{h,\lambda,n}^+ & = & \left(
\begin{array}{ccccccc}
a_{m_o} & \ldots & a_0 & 0 & 0 &\ldots & 0\\
a_{m_o+1} & \ldots & a_1 & a_0 & 0 & \ldots & 0\\
\ldots & \ldots & \ldots & \ldots & \ldots & \ldots & \ldots\\
a_{m_o+n} &\ldots &\ldots & \ldots & \ldots& \ldots &  a_0
\end{array}
\right) & \text{if } \lambda \neq \infty,\\
& & &\\

M_{h,\lambda,n}^+ & = & \left(
\begin{array}{cccccccc}
0 &0 & \ldots & 0 & 0 & 0 &\ldots & 0 \\
0 & a_{m_o} & \ldots & a_1 & a_0 & 0 & \ldots & 0 \\
\ldots & \ldots & \ldots & \ldots & \ldots & \ldots & \ldots & \ldots \\
0 & a_{m_o+n-1} &\ldots &\ldots & \ldots & \ldots& \ldots &  a_0 
\end{array}
\right) & \text{if } \lambda=\infty.
\end{array}$$

We let $M^{+}_{\lambda,n}$ be the block matrix obtained by replacing
in $Mat$ each of its coefficient equal to a rational function $h$ by
$M_{h,\lambda,n}^+$.
\end{definition}

\begin{definition}
  We keep the same notations as in the preceding definition. Let
  $\lambda' \neq \lambda \in K$. Let $h$ be a rational function in the
  indeterminate $t$. Write
  $$S_{h,\lambda}=a_o(t-\lambda)^{-m_o}+a_1(t-\lambda)^{-m_o+1}+\ldots+a_{m_o+n}(t-\lambda)^{n}+\ldots$$
  For $r \geq 0$ an integer, let $\eta
  ^r_0,\eta^r_1,\ldots,\eta^r_{n}$ be the $n+1$ first coefficients of
  the expansion of $\sum^r_{\ell=0} a_\ell(t-\lambda)^{-m_o+\ell}$
  around $\lambda'$ (it is a power series).  Let
  $M_{h,\lambda,n}^{-,\lambda'}$ be the matrix of dimension
  $(n+1,m_o+n+1)$ given by
$$\begin{array}{cccc}
M_{h,\lambda,n}^{-,\lambda'} & = & -\left(
\begin{array}{cccccc}
\eta^{m_o-1}_0 & \ldots & \eta^0_0 & 0 & \ldots & 0\\
\ldots & \ldots & \ldots & \ldots & \ldots &\ldots\\
\eta^{m_o-1}_{n} & \ldots &\eta^0_{n} & 0 & \ldots& 0
\end{array}
\right) & \text{if }\lambda \neq \infty,\\
& & & \\

M_{h,\lambda,n}^{-,\lambda'} & = & -\left(
\begin{array}{cccccc}
\eta^{m_o}_0 & \ldots & \eta^0_0 & 0 & \ldots & 0\\
\ldots & \ldots & \ldots & \ldots & \ldots &\ldots\\
\eta^{m_o}_{n} & \ldots &\eta^0_{n} & 0 & \ldots& 0
\end{array}
\right)& \text{if }\lambda=\infty.\\
\end{array}$$ 
Let $M^{-,\lambda'}_{\lambda,n}$ be the block matrix obtained by
replacing in $Mat$ each of its coefficient equal to a rational function $h$ by
$M_{h,\lambda,n}^{-,\lambda'}$.
\end{definition}
We remark that in this definition, if $\lambda'=\infty$, then the
first line of $M_{h,\lambda,n}^{-,\lambda'}$ is null, which
corresponds to our convention to set the constant term at the infinity
to zero. This convention implies that the $m_o^{th}$ column of
$M_{h,\lambda,n}^{-,\lambda'}$ is zero if $\lambda=\infty$.

\begin{definition}\label{Mn}
  Still keeping the same notations, for $n>0$ an integer, let $M_n$ be
  the matrix of dimension $(2(1+n)(2g+2),2(m_o+n+1)(2g+2))$ given by
$$M_n=\left(
\begin{array}{cccc}
M^+_{\lambda_1,n} & M^{-,\lambda_1}_{\lambda_2,n} & \ldots & M^{-,\lambda_1}_{\infty,n} \\
M^{-,\lambda_2}_{\lambda_1,n} & M^+_{\lambda_2,n} & \ldots & M^{-,\lambda_2}_{\infty,n} \\
\ldots & \ldots & \ldots & \ldots \\
M^{-,\infty}_{\lambda_1,n} & M^{-,\infty}_{\lambda_2,n} & \ldots & M^+_{\infty,n}
\end{array}
\right).$$
\end{definition}

We have
\begin{lemma}\label{l1}
  Let $Mat$ be defined as above. Let $m_c =m\otimes g_c \in M_c$, then
$$v_{Mat.m\otimes g_c,n-m_o}=M_n.v_{m\otimes g_c,n},$$
where $v$ is defined by (\ref{vector}).
\end{lemma}

\begin{proof}
  We saw in the last section that the action of an overconvergent
  function $h \in B^\dagger_K$ on $g_c \in B_c$ is given by
  multiplying $i_D(h)$ by $\sigma(g_c)$ in the Robba ring
  $\mathcal{R}_B$ and then apply $\sigma \circ p_c$.  This last
  operation can be done by subtracting the sum of all the principal
  parts in all the components of $\sigma(B_c)$.

  Now, the matrix $M^+_{\lambda,n}$ is such that for any $m_c \in M_c$
  its product with the local component vector
  $v^{\lambda,\vee}_{m_c,n}$ gives the first $n+1$ terms of the
  analytic parts of the local product $Mat.m_\lambda$. In the same
  manner, the matrix $M^{-,\lambda '}_{\lambda,n,m_o+1}$ is such that
  its product with the local component vector
  $v^{\lambda,\vee}_{m_c,n}$ gives the first $n+1$ terms of the
  expansion locally around $\lambda '$ of the opposite of the
  principal part in $\lambda$ of the local product $Mat.m_\lambda$.
\end{proof}

\subsection{Global method}\label{globm}

We can use the notations introduced in Section \ref{glob} to rewrite the
differential equation (\ref{eqcon}). Here, we let $Mat$ be the matrix
$Mat_{\nabla_{GM}}$ of the Gauss-Manin connection for the basis $\{1,
Y\}$ of $M_c$ as a $B^\dagger_K$-module, so that we have $m_o=1$.

\begin{definition}\label{def1}
Let $n$ be a positive integer. Let $D_n$ be the matrix with dimension
$(2(1+n)(2g+2),2(n+2)(2g+2))$ given by 
$$D=\left(
\begin{array}{ccccc}
D_{\lambda_1,n} & 0 & \ldots & \ldots& 0 \\
0 & D_{\lambda_2,n} & 0 & \ldots & 0\\
\ldots & \ldots & \ldots & \ldots & \ldots\\
0 & \ldots & \ldots & \ldots & D_{\infty,n}\\
\end{array}
\right),
$$
where $D_{\lambda_i,n}$ is the diagonal block matrix with dimension
$(2(n+1),2(n+2))$ such that the diagonal blocks are given by the
$(n+1,n+2)$-matrices
$$\tilde{D}_{\lambda_i,n}=\left(
\begin{array}{ccccccc}
0 & 1 & 0 & \ldots & \ldots & \ldots & 0 \\
0 & 0 & 2 & 0 & \ldots & \ldots\\
\ldots & \ldots & \ldots & \ldots & \ldots & \ldots & \ldots\\
0 & \ldots & \ldots & 0 & n+1& 0 & \ldots \\
\end{array}
\right),$$
and where $D_{\infty,n}$ is the block diagonal $(2(n+1),2(n+2))$
matrix the blocks of which are all equal to the $(n+1,n+2)$ matrix
$$\tilde{D}_{\infty,n}=\left(
\begin{array}{ccccccc}
0 & \ldots & \ldots & \ldots & \ldots & \ldots & 0\\
0 & 0 & \ldots & \ldots & \ldots & \ldots & 0\\
0 & -1 & 0 & \ldots & \ldots & \ldots & 0\\
0 & 0 & -2 & 0 & \ldots & \ldots & 0\\
\ldots & \ldots & \ldots & \ldots & \ldots & \ldots & \ldots \\
0 & \ldots & \ldots & 0 & -(n-1) & 0 & \ldots \\
\end{array}
\right).$$
\end{definition}
We have the
\begin{proposition}
  Let $n>1$ be an integer. Let $m_c \in M_c$ be a solution of
  Equation (\ref{eqcon}). Then the vector $v_{m_c,n}$ is a solution of
  the linear system $(M_n+D_n).v=0$.
\end{proposition}
\begin{proof}
  By definition of $\nabla_c$, this is a consequence of Lemma \ref{l1}
  and of the fact that the matrix $D_n$ is such that $v_{m \otimes
    \frac{\partial}{\partial t}g_c , n-1}=D_n v_{m \otimes g_c,n}$.
\end{proof}

\subsection{Local method}\label{loc}

Let $\lambda \in \Lambda$. To simplify the exposition, we suppose in
the following that $\lambda \neq \infty$. The case $\lambda=\infty$
can be treated exactly in the same way. In the following,
$K((t-\lambda))$ denotes the field of Laurent series in the
indeterminate $(t-\lambda)$.

The local method rests on the remark that Equation (\ref{eqcon})
locally at $\lambda$ can be regarded as a classical inhomogeneous
differential equation if we know enough terms of the global solution.
This is the content of

\begin{proposition}\label{eqinhom}
  Let $m_c=(m_{\lambda_1},\ldots,m_{\lambda_{2g+1}},m_{\infty}) \in
  H^1_{MW,c}(V,\pi_* A^\dagger_K)$. For $\lambda \in \Lambda$, let
  $\nabla_{GM,\lambda}$ be the action of the Gauss-Manin connection on
  the local component in $\lambda$ of an element of $M_c$. For all
  $\lambda \in \Lambda$, there exists a unique $u=u_0+Yu_1$, with
  $(u_0,u_1) \in (K((t-\lambda)))^2$ such that $m_\lambda$ is a
  solution of a non-homogeneous differential equation:
\begin{equation}\label{leqinh}
\frac{\partial}{\partial t}m_\lambda+\nabla_{GM,\lambda}m_\lambda=u.
\end{equation}\end{proposition}
\begin{proof}
Let $m_c=(m_{\lambda_1},\ldots,m_{\lambda_{2g+1}},m_{\infty}) \in
H^1_{MW,c}(V,\pi_* A^\dagger_K)$. By definition, it satisfies the equation
$$\nabla_c(m_c)=0.$$
By rewriting each $m_\lambda$, $\lambda \in \Lambda$, as $$m_\lambda =
\sum_{i=0,1} Y^i \otimes g^\lambda_i,$$ where $g^{\lambda}_i \in
\tilde{R}_{\lambda,c}$ we have
\begin{equation*}
\begin{split}
\nabla_c (m_c)  = \left( \sum_{i=0,1} \nabla_{GM}(Y^i) \otimes
  g^{\lambda_1}_i+ \sum_{i=0,1} Y^i \otimes
  \frac{\partial}{\partial t}g^{\lambda_1}_i, \ldots, \right. \\ \left. \sum_{i=0,1} \nabla_{GM}(Y^i) \otimes
  g^{\infty}_i+ \sum_{i=0,1} Y^i \otimes
  \frac{\partial}{\partial t}g^{\infty}_i\right) dt.
\end{split}
\end{equation*}
Let $f_i \in B^\dagger_K$ be such that
\begin{equation*}
\begin{split}
\nabla_c (m_c)  = \left( \sum_{i=0,1} f_i. Y^i \otimes
  g^{\lambda_1}_i+ \sum_{i=0,1} Y^i \otimes
  \frac{\partial}{\partial t}g^{\lambda_1}_i,\ldots, \right. \\ \left. \sum_{i=0,1\
} f_i.Y^i \otimes
  g^{\infty}_i+ \sum_{i=0,1} Y^i \otimes
  \frac{\partial}{\partial t}g^{\infty}_i\right) dt.
\end{split}
\end{equation*}

Set $$u=\sum_{i=0,1} Y^i \otimes \phi_\lambda 
\left( \sum_{\lambda' \in \Lambda} Pr_{\lambda '}(\phi_{\lambda
    '}(f_i).g^{\lambda '}_i)\right),$$
where $\phi_\lambda$ is defined by (\ref{philambda}) and
$Pr_{\lambda}$ is given by Definition \ref{principal}.
We have
$$\frac{\partial}{\partial t}m_\lambda-\nabla_{GM,\lambda}m_\lambda=u,$$
by definition of $p_c$. This concludes the proof of the proposition.
\end{proof}

Let $m_c=(m_{\lambda_1},\ldots,m_{\lambda_{2g+1}},m_{\infty}) \in
H^1_{MW,c}(V,\pi_* A^\dagger_K)$. For $\lambda \in \Lambda$, write
$$m_\lambda = \sum_{i=0,1} Y^i \otimes g^\lambda_i,$$
where $g^{\lambda}_i \in \tilde{R}_{\lambda,c}$. Let $\Phi$ be the
morphism of $K$-vector spaces defined by $\Phi: H^1_{MW,c}(V,\pi_* A^\dagger_K)
\rightarrow (K[Y])^{2g+2}$,
$$m_\lambda = \sum_{i=0,1} Y^i \otimes g^\lambda_i \mapsto
(\sum_{i=0,1} Y^i \otimes g_i^{\lambda_1}(0), \ldots ,
\sum_{i=0,1} Y^i \otimes g_i^{\lambda_{2g+1}}(0),0).$$ 

\begin{corollary}\label{initial}
The map $\Phi: H^1_{MW,c}(V,\pi_* A^\dagger_K) \mapsto (K[Y])^{2g+2}$
is injective.
\end{corollary}
\begin{proof}
  Let $m_c=(m_{\lambda_1},\ldots,m_{\lambda_{2g+1}},m_{\infty}) \in
  H^1_{MW,c}(V,\pi_* A^\dagger_K)$. By Proposition \ref{eqinhom}, for
  all $\lambda \in \Lambda$, $m_\lambda$ satisfies a degree $1$ local
  inhomogeneous differential equation and there is a unique solution
  of this equation with constant term given by $\Phi(m_c)$.
\end{proof}

We have to check that the solution that we obtain satisfies the
condition that the constant term at the infinity point is zero.
\begin{definition}
Let $n>0$ be an integer. Let $Rel_n$ be the sub-K-vector space of
$K^{2(2g+2)(n+2)}$ generated by the vectors
$$(e_1,\ldots,e_{2(2g+2)(n+2)})^\vee$$ such that the $e_i$ are
zero for $$i \in \{2((2g+1)(n+2)+1+j(n+2)|j \in
\{0,1\}\}.$$ 
\end{definition}

The linear system we associated to the equation $$\nabla_c(m_c)=0$$ admits
trivial solutions that we have to put aside. These solutions come from
the fact that the derivation with respect to $t$ increases the degree
locally at the point at infinity. For instance, when truncating above the
degree $n>0$, then $(0,\ldots,0,1\otimes t^{-n})$ corresponds always to a solution of
the linear system even though in general it is not a solution of the
equation $\nabla_c(m_c)=0$.

\begin{definition}\label{Triv}
Let $n>0$ be an integer. Let $Triv_n$ be the sub-K-vector-space of
$K^{2(2g+2)(n+2)}$ spanned by the vectors
$$(e_1,\ldots,e_{2(2g+2)(n+2)})^\vee$$ such that the $e_i$ are 
zero, except for $$i \in \{2(2g+1)(n+2)+j(n+3-s)|j \in
\{1,2\}\text{, }s \in \{1,2\}\}.$$ 
\end{definition}

In order to get rid of the terms of $Triv_n$, we simply truncate the
solutions and keep only the resulting independent vectors with
coefficients in $K$. It would be enough to truncate only the local
part at infinity, but we truncate globally for the sake of clarity.
Note also that by Corollary \ref{initial}, a solution with all local
constant terms equal to zero is globally equal to zero, so that we
don't drop any valid solution by truncating.

\begin{definition}
If $n>2$ and if $v=(v_1,\ldots,v_{2(2g+2)(n+2)})$ is an element of
$K^{2(2g+2)(n+2)}$ we let
$$Tronc(v,n)=(v_1,\ldots,v_{n-2},v_{n+1},\ldots,v_{2n-2},v_{2n+1},\ldots,v_{2(2g+2)(n+2)-2})^\vee.$$
In particular, we have
$$Tronc(v_{m_c,n},n)=v_{m_c,n-2}.$$
\end{definition}
 
This way, we have
\begin{proposition}\label{soltronc}
  Let $n>0$ be an integer. Let $v$ be a solution of the linear system
$$(M_{n+2}+D_{n+2})v=0$$ in $Rel_n$. Then there exists a
unique element $m_c \in \oplus_\Lambda M_{c,\lambda} \otimes
K[[t-\lambda]]$ solution of the Equation (\ref{eqcon}) such that 
$$v_{m_c,n}=Tronc(v,n)$$ where we have generalized in an evident manner
the definition $v_{m,n}$ to $\oplus_\Lambda M_{c,\lambda} \otimes
K[[t-\lambda]]$. More precisely, there exists 
$v_{Triv} \in Triv_{n+2}$ such that  $v_{m,n+2}=v+v_{Triv}$.
\end{proposition}
\begin{proof}
From Corollary \ref{initial}, a solution is uniquely determined
by its first terms.
\end{proof}

An approximation of a formal solution of a differential equation can
be computed by one or the other of the above explained methods, but
nothing has been said, up to now, about the convergence of the solutions
that we approximate. We have yet to prove that a formal solution of
the differential equation is in $M_c$. This can be done by finding
some conditions for the convergence on the coefficients of a formal
solution.

\subsection{Explicit upper bound of the coefficients of a basis of the
  cohomology}
We present an upper bound on the valuation of the coefficients of a
basis of the space $H^1_{MW,c}(V)=H^1_{MW,c}(U,\pi_*A^\dagger_K)$.
In this section, $v_p$ is the usual $p$-adic valuation on
$K$.

\begin{proposition}\label{meminhom}
  Let $m_c=(m_{\lambda_1},\ldots,m_{\lambda_{2g+1}},m_{\infty})$ 
  be an element of $H^1_{MW,c}(U,\pi_*A^\dagger_K)$. Let $u=u^0+u^1.Y$, 
  with for $j=0,1$, $u^j=\sum_{\ell=0}^\infty u_\ell^j (t-\lambda)^\ell$. 
  Suppose that $m_\lambda$ is a solution of the equation
$$\frac{\partial}{\partial t}m_\lambda+ \nabla_{GM,\lambda}m_\lambda=u,$$
then there exists a real number $B>0$ such that for $j=0,1$ and all
$\ell$
$$v_p(u^j_\ell)>B.$$
\end{proposition}
\begin{proof}
Let $m_c=(m_{\lambda_1},\ldots,m_{\lambda_{2g+1}},m_{\infty}) \in
H^1_{MW,c}(V,\pi_* A^\dagger_K)$ with 
$$m_{\lambda_i}=\sum_{j=0,1} Y^j \sum_{\ell=0}^{\infty}
b^{\lambda_i}_{j,\ell} (t - \lambda_i)^\ell,$$ with
$b^{\lambda_i}_{j,\ell}\in K$ and with $b^{\infty}_{j,0}=0$ (still
keeping the convention $t-\infty=t^{-1})$. Let $\lambda \in \Lambda$.
Writing the Gauss-Manin connection as a quotient:
 $$\nabla_{GM}=\frac{G(t)}{\Delta(t)},$$
 where $G(t)$ is a linear transformation which can be written in the
 basis $\{ 1, Y \}$ as a $(2,2)$-matrix with coefficients in $W(k)[t]$
 and $\Delta(t)\in W(k)[t]$ has simple roots. By Proposition
 \ref{eqinhom}, and because each $\lambda' \in \Lambda_0$ is at most a
 simple root of $\Delta$, the vector $m^\lambda$ satisfies the
 equation
$$\frac{\partial}{\partial t}m^\lambda + \nabla_{GM, \lambda}m^\lambda = u,$$
with $$u=\sum _{\lambda' \in \Lambda}
\frac{G(\lambda')}{\Delta'(\lambda')} c^{\lambda'}
(t-\lambda')^{-1},$$ where $c^{\lambda'}=b^{\lambda'}_{0,0}+Y b^{\lambda'}_{1,0}$.

We have for all $\lambda,\lambda' \in \Lambda_0$,
$$v_p(\lambda-\lambda')=0$$
by hypothesis. As a consequence, if one writes $u=u^0+Yu^1$,
with for $j=0,1$,
$u^j=\sum u_\ell^j (t-\lambda)^\ell$, then
\begin{equation}\label{leu}
  v_p(u_\ell^j) \geq \min_{\lambda' \in \Lambda}\left(v_p\left( 
\frac{G(\lambda')}{\Delta'(\lambda')} c^
      {\lambda'}\right) \right),
\end{equation}
where we extend $v_p$ on vectors with coefficients in $W(k)$ by taking
the minimum of the valuation of the components.

Equation (\ref{leu}) follows the remark that expanding in a
neighbourhood of $\lambda$, we find
$$(t-\lambda')^{-1}=-\sum_{\ell=0}^\infty
\frac{1}{(\lambda'-\lambda)^{\ell+1}}(t-\lambda)^\ell.$$
As a consequence, we can take 
$$B= \min_{\lambda' \in \Lambda}\left(v_p\left( 
    \frac{G(\lambda')}{\Delta'(\lambda')} c^ {\lambda'}\right)
\right),$$ in the statement of the theorem.
\end{proof}

We have also
\begin{theorem}\label{majc}
Let $m_c=(m_{\lambda_1},\ldots,m_{\lambda_{2g+1}},m_{\infty})$ be an element
of $H^1_{MW,c}(V,\pi_* A^\dagger_K)$ with
$$m_{\lambda_i}=\sum_{j=0,1} Y^j \sum_{\ell=0}^{\infty}
b^{\lambda_i}_{j,\ell} (t - \lambda_i)^\ell,$$
with $b^{\lambda_i}_{j,\ell}\in W(k)$ and with
$b^{\infty}_{j,0}=0$. Let $\lambda \in \Lambda$. 
Then there exist $\alpha \in
\Bbb{R}$ and $\beta \in \Bbb{R}$ such that
\begin{equation}\label{eqmaj}
v_p (b^\lambda_{j,\ell}) \geq -(\alpha \log_p(\ell) + \beta),
\end{equation}
for all $j$ and all $\ell$. Moreover, $\alpha$ and $\beta$ can be made
explicit (see the Remark \ref{remab}).
\end{theorem}

\begin{proof}
  We prove that the hypothesis of the Theorem 2.3.3 of
 ~\cite{ThChat} are verified. By a
  direct computation, we find that the local exponents of
  $Mat_{\nabla_{GM}}$ are in $\{0,-\frac{1}{2}\}$, so that they are
  prepared and the hypothesis $1$ is satisfied.  The hypothesis $4$ is
  already contained in our statement. The hypothesis $2$ can be
  checked by applying the classical Dwork's trick (see for example 
\cite{MR0498577}, Proposition 3.1). In order to be able to apply
  this result, it is necessary to provide the $B^\dagger_K$-module
  $(\pi_*A^\dagger_K,\nabla_{GM})$ with a Frobenius morphism. Fix a
  Frobenius morphism $F$ on $A^\dagger_K$ lifting the $p^{th}$ power
  on $A_k$ such that $F$ sends $X$ over $X^p$ and acting on $K$ as the
  Frobenius automorphism. Then $F$ induces a Frobenius morphism
  $F_{GM}$ over $(\pi_*A^\dagger_K,\nabla_{GM})$ and the Dwork's trick
  applies. In the same manner, we provide the dual module with connection
  $(\pi_* A^\dagger_K,\nabla^\vee)$ with the Frobenius morphism
  $F_{GM}^{-1}$.  The hypothesis $3$ is true from Proposition
  \ref{meminhom}, and the hypothesis $5$ is easily verified in our
  case. The expression of $B$ given in the proof of the
  Proposition \ref{meminhom} gives the theorem.
\end{proof} 

The following proposition proves the correctness of the algorithm
described in Section \ref{thealgo}. Taking back the notation of the
Section \ref{secalgo1},
\begin{proposition}\label{propDWT}
  Let $v$ be a solution of the linear system $(M_n+D_n)v=0$ belonging
  to $Rel_n$ and let $m \in \oplus_\Lambda M_{c,\lambda} \otimes
  K[[t-\lambda]]$ be the unique solution of the Equation (\ref{eqcon})
  such that $v_{m,n}=Tronc(v,n)$. Then $m$ is an element of $M_c$.
\end{proposition}
\begin{proof}
  From the Theorem \ref{majc}, the coefficients of the local part
  $m_\lambda$ of $m$ satisfy the logarithmic bounds of Equation
  (\ref{eqmaj}) for all $\lambda \in \Lambda$. In particular, this
  implies that the $m_\lambda$ are all in $R_{\lambda,c}$. 
\end{proof}

\begin{remark}\label{remab}
One can obtain explicit formulas for the constants $\alpha$ and $\beta$ 
of the theorem. The proof of Theorem 2.3.3 (inspired from methods of 
Alan Lauder) in~\cite{ThChat} gives the expressions 
$$\alpha = 2\alpha ' +2$$
and 
$$\beta = 2\beta ' + 2 \log_p(3)-B$$
the constants $\alpha '$ and $\beta '$ are computed in 
\cite[Note 4.11]{MR2261044}, with the 
following expressions:
$$\alpha' =2(1+\log_p(2))+3$$ and
$$\beta' = \alpha' \log_p(5) + \beta_2 + \beta_3$$
where
$$\beta_2 = 2(1+\log_p(2))+3$$
and
$$\beta_3 = 4(\frac{2}{p-1} +4\log_p(3)+2\log_p(2)).$$
\end{remark}

\begin{remark}
Let us consider the term $B$. Recall that we saw in the proof of Proposition 
\ref{meminhom} that we have (we keep the notations of the proof) 
\begin{equation}
  B \geq \min_{\lambda \in \Lambda}\left(v_p\left(
\frac{G(\lambda)}{\Delta'(\lambda)} c^
      {\lambda}\right) \right).
\end{equation}
where $c^\lambda$ is formed by the constant terms of the element of
the cohomology group we consider so that we can suppose that its
$p$-adic valuation is zero.  Now since the matrix of the connection is
$$Mat_{\nabla_{GM}}=Q(t)^{-1}\left(                                      
\begin{array}{cc}                                                             
0 & 0 \\                                                          
0 & \frac{Q'(t)}{2} 
\end{array}    
\right)$$
the only non-zero term of $\frac{G(\lambda')}{\Delta'(\lambda')}$ is $1/2$ 
and we can suppose $B=0$.
\end{remark}

\section{Description of the algorithm and complexity
  analysis}\label{secalgo2}
In this section, we present an algorithm, based on the results of
Section \ref{secalgo1}, to compute a basis of the Monsky-Washnitzer cohomology
with compact support of an hyperelliptic curve. The algorithm takes as
input:
\begin{itemize}
\item a finite field of odd characteristic $k$,
\item a genus $g$ hyperelliptic curve $\overline{C}_k$ over $k$, given
  by an equation $Y^2=\prod_{i=1}^{2g+1} (X-\overline{\lambda}_i)$,
  with $\overline{\lambda}_i \in k$ distinct,
\item two positive integers $P_1$ and $P_2$,
\end{itemize}
and returns a basis of the space $H^1_{MW,c}(V,\pi_* A^\dagger_K)$
computed with analytic precision $P_1$ and $p$-adic precision $P_2$.

\subsection{An algorithm for the computation of a basis of the cohomology
  of Monsky-Washnitzer of a curve}\label{thealgo}

Denote by $C_k$ an affine plane model of $\overline{C}_k$. We denote
by $t$ the coordinate on $\Bbb{A}^1_k$.

\subsubsection{The set up}
For $i=1, \ldots, 2g+1$, let $\lambda_i \in W(k)$ lifting
$\overline{\lambda}_i$. Let $\overline{C}_K$ be the hyperelliptic curve
over $K$ given by the equation
$$Y^2=Q(X) \quad \text{with} \quad Q(X) =\prod_{i=1}^{2g+1} (X-\lambda_i).$$

Denote by $\infty$ the point an infinity of $\overline{C}_K$.
Keeping the notations of Section \ref{secalgo1}, we let $\Lambda=\{
\lambda_1, \ldots, \lambda_{2g+1}, \infty \}$ and $\Lambda_0 =
\Lambda \setminus \{\infty \}$. Let $V$ be the subvariety of $\Proj^1_K$ whose
geometric point set is the complementary of $\Lambda$, let $U=
\pi^{-1}(V)$ where $\pi$ is the projection along the $Y$-axis. Let
$U_k$ and $V_k$ be respectively $U$ and $V$ modulo $p$.

The algorithm goes through the following $3$ steps.
\subsubsection{Step 1: computation of the connection matrix}\label{matcon}
The Gauss-Manin connection matrix on $A^\dagger_K$ is easily
described. We fix from now on the basis $\{1,Y\}$ of the
$B^{\dagger}_K$-module $A^\dagger_K$. The matrix is given by the
derivation with respect to $Y$ in $A^\dagger_K$ seen as a
$B^\dagger_K$-module.  Since in $\Omega^1_A$ we have $dY =
\frac{Q'(X)}{2Q(X)}Y.dX$, the Gauss-Manin connection matrix over
$A_K^\dagger$ associated to the projection $\pi:U \rightarrow V$ is:
$$Mat_{\nabla_{GM}}=Q(t)^{-1}\left(
\begin{array}{cc}
0 & 0 \\
0 & \frac{Q'(t)}{2}
\end{array}
\right).$$

\subsubsection{Step 2: Computation of the matrix $M_{n}$}
(see Section \ref{glob}) Here $n$ is the analytic precision of the
computation which will be fixed later, depending on whether we use
the local or the global method.

In order to obtain the matrix $M_{n}$, we have to compute for $\lambda
\in \Lambda$ the local development in $\lambda$ of $Q'(t)/Q(t)$ that
we denote by $S_{\nabla, \lambda}(t)$ and for each $\lambda, \lambda'
\in \Lambda$ the local development in $\lambda'$ of the principal part
of $S_{\nabla, \lambda}(t)$.

The development of $Q'(t)$ in $\lambda$ is nothing but the evaluation
$Q'(t+\lambda)$ which can be done using Horner's method or the
Paterson-Stockmeyer algorithm~\cite{MR0314238}. The computation of a
development of $1/Q(t)$ in $\lambda$ can be done by
\begin{itemize}
\item computing a local development $S_{Q,\lambda}(t)$ of $Q(t)$ in
  $\lambda$ using Horner's method;
\item inverting $S_{Q,\lambda}(t)$ using a Newton iteration.
\end{itemize}
The case of $\lambda=\infty$ can be treated in a similar manner.

Then we have to compute the product of the local developments in
$\lambda$ of $Q'(t)$ and $1/Q(t)$ to obtain $S_{\Delta, \lambda}(t)$.

As $S_{\nabla, \lambda}(t)$ can only have simple poles, the computation
of a local development of the principal part of $S_{\nabla,
  \lambda}(t)$ in $\lambda'$ boils down to the computation of an
inverse locally at 
 zero of a term of the form $t+\lambda-\lambda'$ which
can be done by a Newton iteration.

\subsubsection{Step 3: solving the equation $\nabla_c(m_c)=0$}
The next step is to solve the equation $\nabla_c(m_c)=0$ on $M_c$.  We
have to compute modulo $(t-\lambda)^{P_1}$ locally at $\lambda \in
\Lambda$ and modulo $p^{P_2}$ for $p$-adic precision. In section
\ref{secalgo1}, we have given two ways to obtain a basis of solutions
of the differential equation $\nabla_c(m_c)=0$. We use the local
method after determining the first terms of a basis thanks to the
global method. It should remarked that due to the special form of the
connection associated to a hyperelliptic curve, it is possible in the
case we consider to compute directly these terms. Still we present the
global method for its general interest.

\paragraph{Global method} We first use the global method to compute a
basis of solutions at small fixed analytic precision. For this, we
have to compute the matrices $M_1$ and $D_1$ and solve the linear
system
$$(M_{1}+D_{1})v=0,$$
over $K$. Then it is necessary to put aside the trivial solutions
belonging to $Triv_{1}$ and project the remaining ones onto $Rel_{1}$. To
conclude, we truncate the remaining vectors as explained in Proposition
\ref{Triv}. 

\paragraph{Local method} Denote by $m^1_c, \ldots, m^{4g+1}_c \in M_c$
the elements of a basis of the space $H^1_{MW,c}(V, \pi_* A_K^\dagger)$ computed
up to analytic precision $1$ with the global method. For
$j=1, \ldots , 4g+1$, we write, $m^j_c=(m^j_{\lambda_1}, \ldots,
m^j_{\lambda{2g+1}}, m^j_\infty)$.

For a fixed $\lambda \in \Lambda$, we explain how to lift
$m^j_\lambda$, for $j=1, \ldots , 4g+1$, using the local differential
equation provided by Proposition \ref{eqinhom}. For this, we have to
compute the constant term of Equation (\ref{leqinh}). The general expression
of this coefficient is
$$u_j=\sum_{i=0,1} Y^i \otimes \phi_\lambda 
\left( \sum_{\lambda' \in \Lambda} Pr_{\lambda '}(\phi_{\lambda
    '}(f_i).g^{\lambda '}_{j,i})\right),$$ where $f_i$ depends only on
the Gauss-Manin connection and for $\lambda \in \Lambda$,
$g_{j,i}^\lambda \in \tilde{R}_{\lambda,c}$ is such that
$$m^j_\lambda =
\sum_{i=0,1} Y^i \otimes g^\lambda_{j,i}.$$
As $f_i$ has only simple poles, we can write
$$\phi_\lambda 
\left( \sum_{\lambda' \in \Lambda} Pr_{\lambda '}(\phi_{\lambda
    '}(f_i).g^{\lambda '}_{j,i})\right)= \sum_{\lambda' \in \Lambda}
\phi_\lambda(Pr_{\lambda'}
(\phi_{\lambda'}(f_i))).g_{j,i}^{\lambda'}(0).$$ We remark that
$g_{j,i}^{\lambda'}(0)$ can be computed with the global method. As a
consequence, once we have computed for a fixed $\lambda \in \Lambda$,
$r_{\lambda,\lambda'}= \phi_\lambda (Pr_{\lambda'}
(\phi_{\lambda'}(f_i)))$, it is possible to recover $u_j$ by computing
a linear combination of the $r_{\lambda,\lambda'}$ with coefficient in
$K$.

Equation (\ref{leqinh}) can be rewritten as an equation
of the form
$$Z'=AZ+B,$$
where $A$ (resp. B) is a $(2,2)$-matrix (resp. $(2,1)$-matrix) with
coefficients in $K[[t]]$. It is possible to compute an approximation
of the unique solution $Z$ of this equation satisfying $Z(0)=v$ with
precision $P_1$ using an asymptotically fast algorithm such as given
by Theorem 2 of~\cite{BostanCOSSS07}.  The initial value $v$ comes
from the global method.

\subsection{Complexity analysis}\label{complex1}
In order to assess the complexity of our algorithm we use the
computational model of a Random Access Machine~\cite{MR1251285}.  In
this paper, we use the soft-O notation and choose to ignore logarithmic
terms in the complexity functions. For instance, using the algorithm
of Sch\"onhage-Strassen, the multiplication of two $n$-bit length
integers takes $\tilde{O}(n)$ time.  We suppose that $k$ is a finite
field of cardinality $q$ and characteristic $p$.  Let $x \in W(k)$, we
say that we have computed $x$ up to precision $P_2$ if we have
computed a representative of $x \mod p^{P_2}$.  In the following we
assume the sparse modulus representation which is explained in
\cite[pp.239]{MR2162716}.  Let $x,y \in W(k) \slash p^{P_2} W(k)$,
under this assumption one can compute the product $xy$ with precision
$P_2$ by performing $M=\tilde{O}(\log (q) P_2)$ bit operations. The
storage requirement for an element of $W(k)$ with precision $P_2$ is
$O(\log(q) P_2)$.

Let $h=\sum_{\ell\geq 0} a_\ell t^\ell \in W(k)[[t]]$ with $a_\ell \in
W(k)$.  We say that we have computed $h$ up to precision $P_1$ if we
have computed a representative of $h \mod t^{P_1}$. Using the
algorithm given in~\cite{MR2001757}, the multiplication of two
polynomials of degree $P_1$ with coefficients in $W(k)$ takes
$\tilde{O}(P_1)$ operations in $W(k)$.  As a consequence, the
multiplication of two elements of $W(k)[[t]]$ with precision $P_1$
takes $N=\tilde{O}(\log (q) P_2P_1)$ time. The storage requirement for
an element of $W(k)[[t]]$ with analytic precision $P_1$ and $p$-adic
precision $P_2$ is $O(\log(q)P_1 P_2)$.

Now, we give time and memory complexity bounds for the computation of
a basis of the cohomology with analytic precision $P_1$ and $p$-adic
precision $P_2$. We refer to Section \ref{thealgo} for the description
of each step.

\subsubsection{Step 1:} The asymptotic running time of this step is
clearly negligible with respect to the other steps.

\subsubsection{Step 2:}
Using the local method, we only have to compute $M_1$ which makes the
running time of this step also negligible with respect to the other
steps.

\subsubsection{Step 3:} In this step, we use the global method to
compute the first terms of the solutions required for the local method.

\paragraph{The global method} We have to inverse a matrix with
coefficients in $K$ the dimension of which is in the order of $g$.
The total cost is $\tilde{O}(g^3 \log (q) P_2)$ time and $O(g^2
\log(q) P_2)$ memory.

\paragraph{The local method} We keep the notations of Section
\ref{thealgo} Step 3. First, for a fixed $\lambda \in \Lambda$, we
give the running time and memory usage for the computation of a lift
of $m_j^\lambda$ for $j$ running in $\{ 1, \ldots , 4g+1 \}$.

We compute $r_{\lambda,\lambda'}= \phi_\lambda (Pr_{\lambda'}
(\phi_{\lambda'}(f_i)))$ for $\lambda' \in \Lambda$.  In order to do
this, we have to develop $f_i$ in $\lambda'$. With our hypothesis the
only non trivial $f_i$ has the form $Q'(t)/Q(t)$.

The computation of a local development of $Q'$ in $\lambda'$ can be
done with the evaluation $Q'(t+\lambda')$. Using Horner's method or
the Paterson-Stockmeyer algorithm~\cite{MR0314238}, it takes $O(g M)$
time. The computation of a local development of $1/Q(t)$ in $\lambda'$
with a Newton iteration at the expense of $O(\log(P_1) N)$ time.  Then
we have to compute the product of the local developments of $Q'$ and
$1/Q$.  This product takes $O(N)$ time. Taking the principal part is
trivial. Then use again a Newton iteration to compute a development in
$\lambda$ of a principal part of the form $1/(t-\lambda')$. We are done
for the computation of $r_{\lambda,\lambda'}$.

We have to repeat this operation $O(g)$ times to obtain all the
coefficients $r_{\lambda, \lambda'}$ at the expense of $\tilde{O}(g
\log(q) P_1 P_2)$ time and $\tilde{O}(g \log(q) P_1 P_2)$ memory.

Then in order to lift $m_\lambda^j$, for $j=1, \ldots, 4g+1$, we have
to solve an equation of the form $Z'=AZ+B$. This is a $2$ dimensional
linear differential equation of order $1$ and applying Theorem 2 of
\cite{BostanCOSSS07}, a solution of this equation with analytic
precision $P_1$ can be computed in $O(N)$ time.

Now, we have to repeat all the preceding operations for $\lambda$
running in $\Lambda$. In all the computational time is $\tilde{O}(g^2 \log(q) P_1
P_2)$ and the memory consumption is $O(g \log(q) P_1 P_2)$

\begin{proposition}
  Let $P_1$ and $P_2$ be positive integers. The global time for the
  computation of a basis of $H^1_{MW,c}(V, \pi_* A_K^\dagger)$ with
  analytic precision $P_1$ and $p$-adic precision $P_2$ is bounded by
  $\tilde{O}(g^2  \log (q) P_1P_2)$. The memory consumption is $O(g
  \log (q) P_1 P_2)$.
\end{proposition}
 
\subsection{An example}
In this section, we give a detailed example of computation. We
consider the case of the elliptic curve with equation $Y^2=X^3-X$ over
$\Bbb{F}_5$. We lift the equation to $Y^2-X^3+X \in \Bbb{Z}_5[X,Y]$ so
that we are interested in the finite module $M$ with basis $\{1,Y\}$
over $\Bbb{Q}_5[t,t^{-1},(t-1)^{-1},(t+1)^{-1}]^\dagger$. We denote as
before $\Lambda = \{\infty, 0,1,-1\}$ and consider the connection on
$M$ given by the matrix
$$ Mat_{\nabla_{GM}} = (\Delta(t)) ^{-1} \left( 
\begin{array}{ccc} 0 & 0 \\
             0 & \frac{3t^2-1}{2} \end{array} \right),$$
where $\Delta(t)=t^3-t$.
Let $h(t)=1/(2\Delta) (3t^2-1)$ be the only non-zero term of this matrix. 
Its local development in Laurent series at the elements of $\Lambda$
are:

\begin{itemize}
\item at $\infty$: $\frac{3}{2}t +t^3 + O(t^4)$,
\item at $0$: $\frac{1}{2} t^{-1} - t - t^3 + O(t^4)$,
\item at $1$: $\frac{1}{2} t^{-1} + \frac{11}{8} - \frac{5}{8} t + \frac{9}{16} t^2 - \frac{17}{32} t^3 + O(t^4)$,
\item at $-1$: $\frac{1}{2} t^{-1} - \frac{11}{8} - \frac{5}{8} t - \frac{9}{16} t^2 - \frac{17}{32} t^3 + O(t^4)$.
\end{itemize}
So that we have:
$$ M^+_{h,\infty,1} = \left( 
\begin{array}{ccc} 
  0 & 0 & 0\\
  0 & 0 & 0 
 \end{array} \right),$$
 $$ M^+_{h,0,1} = \left( 
\begin{array}{ccc} 
  0 & \frac{1}{2} & 0\\
  -1 & 0 & \frac{1}{2} 
 \end{array} \right),$$
$$ M^+_{h,1,1} = \left( 
\begin{array}{ccc} 
  \frac{11}{8} & \frac{1}{2} & 0\\
  -\frac{5}{8} &  \frac{11}{8} & \frac{1}{2} 
 \end{array} \right),$$
$$ M^+_{h,-1,1} = \left( 
 \begin{array}{ccc} 
  -\frac{11}{8} & \frac{1}{2} & 0\\
  -\frac{5}{8} &  -\frac{11}{8} & \frac{1}{2} 
 \end{array} \right).$$

Since $h$ has no pole at infinity, the matrices
$M^{-,\lambda}_{h_i,\infty,1}$ are zero for $\lambda = 0,1,-1$.  We
expand $\frac{1}{t}$ locally
\begin{itemize}
\item at $\infty$: $t$,
\item at $1$: $1 - t + t^2 - t^3 + O(t^4)$,
\item at $-1$: $-1 - t - t^2 - t^3 + O(t^4)$,
\end{itemize}
which gives:
$$ M^{-,\infty}_{h,0,1} = -\frac{1}{2}\left( 
\begin{array}{ccc} 
  0 & 0 & 0\\
  1 & 0 & 0 
\end{array} \right),$$
 $$ M^{-,1}_{h,0,1} = -\frac{1}{2}\left( 
\begin{array}{ccc} 
  1 & 0 & 0\\
  -1 & 0 & 0 
\end{array} \right),$$
 $$ M^{-,-1}_{h,0,1} = -\frac{1}{2}\left( 
\begin{array}{ccc} 
  -1 & 0 & 0\\
  -1 & 0 & 0 
\end{array} \right).$$

We find similarly
$$ M^{-,\infty}_{h,1,1} = -\frac{1}{2}\left( 
\begin{array}{ccc} 
  0 & 0 & 0\\
  1 & 0 & 0 
\end{array} \right),$$
 $$ M^{-,0}_{h,1,1} = -\frac{1}{2}\left( 
\begin{array}{ccc} 
  -1 & 0 & 0\\
  -1 & 0 & 0 
\end{array} \right),$$
 $$ M^{-,-1}_{h,1,1} = -\frac{1}{2}\left( 
\begin{array}{ccc} 
  -\frac{1}{2} & 0 & 0\\
  -\frac{1}{4} & 0 & 0 
\end{array} \right),$$
and
$$ M^{-,\infty}_{h,-1,1} = -\frac{1}{2}\left( 
\begin{array}{ccc} 
  0 & 0 & 0\\
  1 & 0 & 0 
\end{array} \right),$$
 $$ M^{-,0}_{h,-1,1} = -\frac{1}{2}\left( 
\begin{array}{ccc} 
  1 & 0 & 0\\
  -1 & 0 & 0 
\end{array} \right),$$
 $$ M^{-,1}_{h,0,-1} = -\frac{1}{2}\left( 
\begin{array}{ccc} 
  \frac{1}{2} & 0 & 0\\
  -\frac{1}{4} & 0 & 0 
\end{array} \right).$$
Now we obtain the final matrix
$$\frac{1}{2}.\left( \begin{array}{cccccccccccccccccccccccccccc} 
  0 & 0 & 0 & 0 & 0 & 0 & 0 & 0 & 0 & 0 & 0 & 0 & 0 & 0 & 0 & 0 & 0 & 0 & 0 & 0 & 0 & 0 & 0 & 0\\
  0 & 0 & 0 & 0 & 0 & 0 & 0 & 0 & 0 & 0 & 0 & 0 & 0 & 0 & 0 & 0 & 0 & 0 & 0 & 0 & 0 & 0 & 0 & 0\\
  0 & 0 & 0 & 0 & 0 & 0 & 0 & 0 & 0 & 0 & 0 & 0 & 0 & 0 & 0 & 0 & 0 & 0 & 0 & 0 & 0 & 0 & 0 & 0\\
  0 & 0 & 0 & 0 & 0 & 0 & 0 & 0 & 0 & -1 & 0 & 0 & 0 & 0 & 0 & -1 & 0 & 0 & 0 & 0 & 0 & -1 & 0 & 0\\
  0 & 0 & 0 & 0 & 0 & 0 & 0 & 2 & 0 & 0 & 0 & 0 & 0 & 0 & 0 & 0 & 0 & 0 & 0 & 0 & 0 & 0 & 0 & 0\\
  0 & 0 & 0 & 0 & 0 & 0 & 0 & 0 & 4 & 0 & 0 & 0 & 0 & 0 & 0 & 0 & 0 & 0 & 0 & 0 & 0 & 0 & 0 & 0\\
  0 & 0 & 0 & 0 & 0 & 0 & 0 & 0 & 0 & 0 & 3 & 0 & 0 & 0 & 0 & 1 & 0 & 0 & 0 & 0 & 0 & -1 & 0 & 0\\
  0 & 0 & 0 & 0 & 0 & 0 & 0 & 0 & 0 & -2 & 0 & 5 & 0 & 0 & 0 & 1 & 0 & 0 & 0 & 0 & 0 & 1 & 0 & 0\\
  0 & 0 & 0 & 0 & 0 & 0 & 0 & 0 & 0 & 0 & 0 & 0 & 0 & 2 & 0 & 0 & 0 & 0 & 0 & 0 & 0 & 0 & 0 & 0\\
  0 & 0 & 0 & 0 & 0 & 0 & 0 & 0 & 0 & 0 & 0 & 0 & 0 & 0 & 4 & 0 & 0 & 0 & 0 & 0 & 0 & 0 & 0 & 0\\
  0 & 0 & 0 & 0 & 0 & 0 & 0 & 0 & 0 & -1 & 0 & 0 & 0 & 0 & 0 & \frac{11}{4} & 3 & 0 & 0 & 0 & 0 & -\frac{1}{2} & 0 & 0\\
  0 & 0 & 0 & 0 & 0 & 0 & 0 & 0 & 0 & 1 & 0 & 0 & 0 & 0 & 0 & -\frac{5}{4} & \frac{11}{4} & 5 & 0 & 0 & 0 & \frac{1}{4} & 0 & 0\\
  0 & 0 & 0 & 0 & 0 & 0 & 0 & 0 & 0 & 0 & 0 & 0 & 0 & 0 & 0 & 0 & 0 & 0 & 0 & 2 & 0 & 0 & 0 & 0\\
  0 & 0 & 0 & 0 & 0 & 0 & 0 & 0 & 0 & 0 & 0 & 0 & 0 & 0 & 0 & 0 & 0 & 0 & 0 & 0 & 4 & 0 & 0 & 0\\
  0 & 0 & 0 & 0 & 0 & 0 & 0 & 0 & 0 & 1 & 0 & 0 & 0 & 0 & 0 & \frac{1}{2} & 0 & 0 & 0 & 0 & 0 & -\frac{11}{4} & 3 & 0\\
  0 & 0 & 0 & 0 & 0 & 0 & 0 & 0 & 0 & 1 & 0 & 0 & 0 & 0 & 0 & \frac{1}{4} & 0 & 0 & 0 & 0 & 0 & -\frac{5}{4} & -\frac{11}{4} & 5\\
\end{array}\right).$$

Its kernel is spanned by the vectors:\\

\begin{itemize}
\item $v_1=(1,0,0,0,0,0,0,0,0,0,0,0,0,0,0,0,0,0,0,0,0,0,0,0)^\vee,$
\item $v_2=(0,1,0,0,0,0,0,0,0,0,0,0,0,0,0,0,0,0,0,0,0,0,0,0)^\vee,$
\item $v_3=(0,0,1,0,0,0,0,0,0,0,0,0,0,0,0,0,0,0,0,0,0,0,0,0)^\vee,$
\item $v_4=(0,0,0,1,0,0,0,0,0,0,0,0,0,0,0,0,0,0,0,0,0,0,0,0)^\vee,$
\item $v_5=(0,0,0,0,1,0,0,0,0,0,0,0,0,0,0,0,0,0,0,0,0,0,0,0)^\vee,$
\item $v_6=(0,0,0,0,0,1,0,0,0,0,0,0,0,0,0,0,0,0,0,0,0,0,0,0)^\vee,$
\item $v_7=(0,0,0,0,0,0,1,0,0,0,0,0,0,0,0,0,0,0,0,0,0,0,0,0)^\vee,$
\item $v_8=(0,0,0,0,0,0,0,0,0,0,0,0,1,0,0,0,0,0,0,0,0,0,0,0)^\vee,$
\item $v_9=(0,0,0,0,0,0,0,0,0,0,0,0,0,0,0,0,0,0,1,0,0,0,0,0)^\vee,$
\item $v_{10}=(0,0,0,0,0,0,0,0,0,1,-\frac{1}{3},\frac{3}{5},0,0,0,0,\frac{1}{6},-\frac{29}{24},0,0,0,-1,-\frac{5}{4},-\frac{91}{80})^\vee,$
\item $v_{11}=(0,0,0,0,0,0,0,0,0,0,-\frac{2}{3},0,0,0,0,1,-\frac{13}{12},\frac{43}{48},0,0,0,-1,-\frac{13}{12},-\frac{43}{48})^\vee.$
\end{itemize}

The first six vectors are trivial or non-relevant solutions since
either their truncature is zero either setting their constant term at
the infinity to zero makes them null, so that we find that
$H^1_{MW,c}(V,\pi_*A_K^\dagger)$ has dimension $5$. Since our curve is
of genus $1$ and since we took off three points out of it, this agrees
with the theoretical dimension.

\section{The action of Frobenius on the basis}\label{seclift}
We keep in this section the notations already introduced, and suppose
furthermore that the roots $\lambda$ of $Q$ are Teichm\"uller
elements~\cite{MR2162716}.  We explain how to compute a lifting of the
Frobenius morphism to $M_c$ and obtain its action on a basis of
$H^1_{MV,c}(U, \pi_* A_K^\dagger)$. Of course all the computations are
made with finite analytic and $p$-adic precisions and the key point
here is the determination of sufficient precisions in order to
guarantee the correctness of the final result.

\subsection{Lifting the Frobenius morphism}
Following Kedlaya~\cite{Kedlaya2001}, we define
a lift $F$ of the $p$-th Frobenius on $A^\dagger_K$ by setting
$F(X)=X^{p}$ and
$$F(Y)= Y^{p} \left( 1+ \frac{Q^{\sigma}(X^{p}) - Q(X)^{p}}
  {Q(X)^{p}}\right)^{1/2}$$ where $\sigma$ is the canonical Witt
vectors Frobenius.  The expansion of the square root can be computed
using a Newton iteration (see \cite{Kedlaya2001}). We then have the
\begin{proposition}\label{proplift}
  For all positive integer $n$ the element $F(Y) \mod p^n$ can be
  written as $Y$ times a rational fraction in $X$ such that its
  numerator and denominator are relatively primes, its denominator is
  $Q(X)^d$ with $d \leq p n - \frac{p-1}{2}$ and its pole order at the
  infinity $(2g+1)[p/2]$.
\end{proposition}

\begin{proof}
  The only non obvious fact from the expression of the Newton
  iteration is the bound for $d$. Since we have
$$F(Y)= Y^{p} \left( 1+ \frac{Q^{\sigma}(X^{p}) - Q(X)^{p}} {Q(X)^{p}}\right)^{1/2},$$
we can write
$$F(Y)=Y .Q(X)^{\frac{p-1}{2}} \sum_{k \geq 0} \left(
\begin{array}{c}
1/2\\
k
\end{array}
\right) \frac{p^k E(X)^k}{Q(X)^{pk}},$$ where
$E(X)=1/p.(Q^\sigma(X^p)-Q(X)^p)$ and $E$ having integral
coefficients. In particular, since the binomial coefficients are
$p$-adic integers we have
$$ F(Y) = \sum a_{i,j} \frac{X^i}{Q(X)^j}$$
with
$$v_p(a_{i,j}) \geq j/p+\frac{p-1}{2p}$$
and we are done.
\end{proof}

We obtain an approximation $F_n$ of a $\sigma$-linear endomorphism of
$A^\dagger_K$ lifting the Frobenius endomorphism on $A_k$ up to
$p$-adic precision $n$ by setting $F_n(X)=X^{p}$ and $F_n(Y)$ equal to
the truncated development of the rational fraction obtained with
Proposition \ref{proplift}.

\subsection{Twisted local equation}
In the relative situation that we consider, the Frobenius lifting
decomposes as a Frobenius lift on $B^\dagger_K$, which sends $t$ to
$t^{p}$ that we denote $F_B$ (the 'local' Frobenius), and a
Frobenius on the $B^\dagger_K$-module $A^\dagger_K$ (making it an
$F$-isocrystal). We first consider the computation of the action
of $F_{B}$. 

A direct way to compute the action of $F_B$ on a element $m_c \in M_c$
is to make the evaluation $t \mapsto t^{p}$ in all the local
developments at $\lambda \in \Lambda$, apply $\sigma$ on the
coefficients and develop the result to recover a series in
$(t-\lambda)$.

The following remark leads to a more efficient method.  Let $m_c \in
M_c$ representing an element of a basis of $H^1_{MW,c}(U, \pi_*
A_K^\dagger)$.  We recall that, by Proposition \ref{eqinhom}, a local
component $m_\lambda$ of $m_c$ in $\lambda \in \Lambda$ satisfies a
non-homogeneous differential equation
\begin{equation}
 \frac{\partial}{\partial t}m_\lambda-M_{\nabla,\lambda}m_\lambda=u.
\end{equation}

From this equation, we deduce the
\begin{proposition}\label{localfrob}
  For $\lambda \in \Lambda$, the image of $m_\lambda \in
  \tilde{R}_c^\lambda$ by the local Frobenius $F_B(m_\lambda)$
  satisfies a local differential equation
\begin{equation}
\begin{split}\label{eqlocalfrob}
  & (t^{p}-\sigma(\lambda)) \frac{\partial}{\partial t} F_B(m_\lambda)-p
  t^{p-1} F_B((t-\lambda)M_{\nabla, \lambda}) F_B(m_\lambda) \\ & =p
  t^{p-1}(t^{p} - \sigma(\lambda)) F_B(u).
\end{split}
\end{equation} 
\end{proposition}

Proposition \ref{localfrob} yields a very efficient algorithm to
compute the action of $F_B(m_\lambda)$ given its first terms. Note
that here the assumption that $\lambda$ is a Teichm\"uller lifting is
crucial.

\subsection{Formulas for the theoretical precision}
In this paragraph, we explain how to apply the isocrystal Frobenius.
Here arises a technical difficulty. Since this computation consists in
replacing $Y$ by $F(Y)$, where $F(Y)$ is $Y$ times an element of
$B^\dagger_K$, this operation boils down to the computation of the
action of an overconvergent function in $B_K^\dagger$ (with infinite
negative powers) on an element of $M_c$ (with infinite positive
powers). In order to perform this to a certain precision, we take
advantage of the sharp control we have on the size of the
coefficients of both terms.  The following theorem gives expressions
for sufficient $p$-adic and analytic precisions.

\begin{theorem}\label{th2}
  Let $m_c \in M_c$ be an element of $H^1_{MW,c}(V,\pi_ù A_K^\dagger)$.
  Write $m_c=(m_{\lambda_1},\ldots,m_{\lambda_{2g+1}},{\infty})$ with
$$m_{\lambda_i}=\sum_{j=0,1} Y^j \sum_{\ell=0}^{\infty}
b^{\lambda_i}_{j,\ell} (t - \lambda_i)^\ell,$$
where $b^{\lambda_i}_{j,\ell}\in K$ and $b^{\infty}_{j,0}=0$. Let $\alpha$
and $\beta$ be integers such that
$$v_p(b^\lambda_{j,\ell}) > -(\alpha \log_p \ell+\beta)$$
for all $j$,$\ell$ and $\lambda \in \Lambda$.
Then if we set 
$$n=\max (2\alpha \log_p\left(\frac{\alpha}{2\ln(p)}\right), 2( \alpha+\beta + P_2))$$
and
$$P_1=p n - \frac{p-1}{2}$$
the image by $F_n$ of 
$m_c$ truncated at the degree $P_1$ is equal to the image of $m_c$ by
$F$ modulo $p^{P_2}$
\end{theorem}

\begin{proof}
 If we write
$$F(Y)=\sum_{j=0,1} f_j(X) Y^j,$$
the image by $F$ of $m_c$ is given by the products
$$f_j(t).\sigma(b^\lambda_{j,\ell})(F_B(t-\lambda))^i,$$ that is to say
$$f_j(t).\sigma(b^\lambda_{j,\ell})(t^{p}-\sigma(\lambda))^i$$
for all $\ell$ and $j$.
Now for $n$ positive, we have $F_n(Y)= \sum_{j=0,1}
\tilde{f}_{n,j}(X) Y^j$ with $p^n$ dividing
$f_j(t)-\tilde{f}_{n,j}(t)$, and $\tilde{f}_{n,j}(t)$ has its degree
bounded by $p n - \frac{p-1}{2}$\\ 

Here we have to use the following easy lemma 
\begin{lemma}
  Let $\lambda \in W(t)$ be a Teichm\"uller element. Let
  $$S=\sum_{\ell=0}^{\infty} b_{\ell} (t - \lambda)^\ell$$ with
  $b_{\ell}\in K$ be such that there exist integers $\alpha$ and $\beta$
  with
$$v_p (b_{\ell}) \geq -(\alpha \log_p \ell+\beta)$$
for all $\ell \in \N$. If we write 
$$F(S)=\sum_{\ell=0}^{\infty} \sigma(b_{\ell}) (t^{p} -
\sigma(\lambda))^\ell=\sum_{\ell=0}^{\infty} a_{\ell} (t - \lambda)^\ell$$
 then we have for all $\ell\in \N$
$$v_p (a_{\ell}) \geq -(\alpha \log_p \ell
+\beta).$$
\end{lemma}

In order to prove Theorem \ref{th2} using the preceding lemma, we
have to find $n$ and $P_1$ such that for all $\ell>P_1$ we have
$v_p\left( (f_j(t)-\tilde{f}_{n,\ell}(t)).\sigma(b^\lambda_{j,\ell})
\right) \geq P_2$.

We have to solve the inequation
\begin{equation}
\alpha \log_p(p n - \frac{p-1}{2})+\beta -n \leq -P_2.
\end{equation}
If we set $\beta ' = \alpha +\beta + P_2$, it is sufficient to solve 

\begin{equation}
n-\alpha \log_p n - \beta' \geq 0.
\end{equation}
Now if we have $n/2 \geq \beta'$ and
\begin{equation}\label{eqfin}
n/2-\alpha \log_p n \geq 0
\end{equation}
we are done. Equation (\ref{eqfin}) is true for $n >
-\frac{2\alpha}{\ln(p)} \W(-1,-\frac{2}{\alpha})$ where $\W(-1,.)$ is
the real branch defined on the interval $[-1/e, 0]$ of the classical
Lambert function. In particular, given that
$$-\W(-1,-1/t) < \ln(t)$$
for all $t$ of the interval, the proposition is true. We refer to~\cite{MR1414285}
for a detailed survey on the Lambert function.
\end{proof}

\subsection{Recovering the zeta function}
From the action of the Frobenius morphism on a basis of $H^1_{MW,c}(V,
\pi_* A_K^\dagger)$, it is easy to recover the zeta function of the
curve $\overline{C_k}$ thanks to the Lefschetz trace formula (see
\cite[Cor.6.4]{MR1225991}). In our case this formula reads
$$Z(\overline{C}_k,t)=\frac{\det(1-t \phi_c^1)}{(1-t)(1-qt)}$$
where $\phi^1_c$ is the representation of the Frobenius morphism
acting on $H^1_{MW,c}(\overline{C}/K)$.

By the preceding results of this section, we can compute the matrix
$M_{F}$ of the action of the $p$-power Frobenius on a basis of
the space $H^1_{MW,c}(V,\pi_* A^\dagger_K)$.

We explain how to recover the Zeta function of our initial curve from
it.  We can embed the space $H^1_{MW,c}(C_k/K)$ in $H^1_{MW,c}(U_k/K)$
and thanks to the localization exact sequence in Monsky-Washnitzer
cohomology with compact support (deriving from the one in rigid cohomology, 
see~\cite{MR1728610})
$$0  \rightarrow H^0_{MW,c}((C_k \setminus U_k)/K) \rightarrow 
H^1_{MW,c}(U_k/K) \rightarrow H^1_{MW,c}(C_k/K) \rightarrow 0$$
give a description of a supplement. Namely 
$$\{(1\otimes 1,0,\ldots,0),\ldots,(0,\ldots,1\otimes 1,0)\}$$
in $\oplus_{\lambda \in \Lambda_0} (A_K^\dagger \otimes_{B^\dagger_K} 
\tilde{R}_{\lambda,c} )\oplus (A_K^\dagger \otimes_{B^\dagger_K} 
R_{\infty,c})$ is a basis of such a supplement 
(we identify here 
$H^1_{MW,c}(U_k/K)$ and $H^1_{MW,c}(V_k/K,\pi_*A^\dagger_K)$) that we call 
$W$. 
Let 
$\tilde{M}_{F}$ denote the matrix of the action of the $p$-th Frobenius 
on a basis of a supplement of $W$. Let $n$ be the absolute degree of
$k$ the base field of $C_k$. By computing the product
$$\tilde{M}_{\Sigma}=\tilde{M}_{F}\tilde{M}^{\sigma}_{F}\ldots
\tilde{M}_{F}^{\sigma^{n}},$$
we recover the matrix of the total Frobenius.

\subsection{Description of the algorithm and complexity analysis}\label{complex2}
The computation of the Frobenius representation on $H^1_{MW,c}(V,
\pi_* A^\dagger_K)$ can be done in three steps.

\subsubsection{Step 1: Lift of the Frobenius morphism}
Write 
$$F(Y)= Y Q(X)^{(p-1)/2} \left( 1+ \frac{Q^{\sigma}(X^{p}) - Q(X)^{p}}
  {Q(X)^{p}}\right)^{1/2},$$ and for each $\lambda \in \Lambda$,
compute a local analytic development up to precision $P_1$ of the
square root using a Newton iteration as in~\cite{Kedlaya2001}.

The dominant step of this operation is the Newton iteration which
takes $\tilde{O}(\log (q) P_1P_2)$ time and has to be repeated $O(g)$
times for the total cost of $\tilde{O}(g \log(q) P_1 P_2)$.

\subsubsection{Step 2: Computation of the representation of the
  Frobenius morphism}
Denote by $m^1_c, \ldots, m^{4g+1}_c \in M_c$ the elements of a basis
of the space $H^1_{MW,c}(V, \pi_* A_K^\dagger)$ computed with analytic
precision $1$. For $j=1, \ldots , 4g+1$, we write,
$m^j_c=(m^j_{\lambda_1}, \ldots, m^j_{\lambda{2g+1}}, m^j_\infty)$.

Next, for a fixed $\lambda \in \Lambda$, we do the following
operations:

\begin{enumerate}
\item For $j=1, \ldots, 4g+1$, compute the action of the local
  Frobenius $F_B(m^j_\lambda)$ up to the analytic precision $P_1$
  using the local differential equation given by Proposition
  \ref{localfrob}.
\item In the expression of $m^j_\lambda$, replace $Y$ by $\nabla_{GM,
    \lambda}(Y)=Y \sum_{\ell=-d_0}^{P_1} a_\ell t^\ell$ where
  $\sum_{\ell=-d_0}^{P_1} a_\ell t^\ell$ is the local expression in
  $\lambda$ of the lift of the relative Frobenius morphism obtained in
  Step 1. Then develop to obtain ${m'}^j_\lambda=F_B(m_\lambda^j)$.
\end{enumerate}

For the first operation, we have to compute the constant term
$F_B(u_j)$ of Equation (\ref{eqlocalfrob}) associated to
$m^j_\lambda$.  Keeping the notations of Section \ref{thealgo} Step 3,
we have
$$F_B \big( \sum_{\lambda' \in \Lambda}
\phi_\lambda(Pr_{\lambda'}
(\phi_{\lambda'}(f_i))).g_{j,i}^{\lambda'}(0) \big)=\sum_{\lambda' \in
  \Lambda} F_B \big(
\phi_\lambda(Pr_{\lambda'}
(\phi_{\lambda'}(f_i))) \big) .g_{j,i}^{\lambda'}(0)^\sigma.$$

For $\lambda' \in \Lambda$, we have to compute the action of $F_B$ on
principal parts of the form $1/(t-\lambda')$ and develop in $\lambda$.
The first thing can be done by computing a local development in
$\lambda$ of $t^p-\lambda'$ and then use a Newton iteration to
inverse the result.  These operations take $\tilde{O}(\log(q) P_1
P_2)$ time and has to be repeated $O(g)$ times with $O(g\log(q) P_1
P_2)$ memory consumption.

Using the asymptotically fast algorithm provided by Theorem 2
of~\cite{BostanCOSSS07} the total amount of time for solving $O(g)$
equations is $\tilde{O}(g \log(q) P_1 P_2)$. 

The second step is just $O(g)$ products of series with analytic precision
$P_1$ which takes $\tilde{O}(g \log(q) P_1 P_2)$.

For $\lambda$ running in $\Lambda$, all the preceding operations
allows us to recover ${m'}^j_\lambda$ with analytic precision $P_1$ for
$j=1, \ldots , 4g+1$ and $\lambda \in \Lambda$ for $\tilde{O}(g^2
\log(q) P_1 P_2)$ time and $O(g \log(q) P_1 P_2)$ memory
consumption.

The next thing to do is for $j=1, \ldots , 4g+1$ and for $\lambda \in
\Lambda_0$, subtract the principal part of ${m'}^j_\lambda$ to
${m'}^j_\infty$ to recover $F_\infty(m^j_\infty)$.  In order to
compute the contribution of the principal part of ${m'}^j_\lambda$ in
$\infty$, we have to obtain a local development in $\infty$ of an
element of $R_\lambda$. By considering a Laurent series in $\lambda$
as an analytic series in $\lambda$ times a term of the form
$1/(t-\lambda)^{m_o}$, we have to compute a local development in
$\infty$ of an analytic series $S_\lambda(t)$ and on the other side of
a term of form $1/(t-\lambda)^{m_o}$ and then take the product.  The
local development in $\infty$ of $S_\lambda(t)$ with precision $P_1$
can be done by computing the evaluation $S_\lambda(1+\lambda t)$ which
can be decomposed into the evaluation of $S_\lambda(1+t')$ using the shift 
operator for polynomials described in \cite{MR0398151} and the
substitution $t'=\lambda t$. This can be done in $\tilde{O}(\log(q)
P_1 P_2)$ time.  To compute a development $1/(t-\lambda)^{m_o}$ in $\infty$
we have to compute a development of $1/(1+\lambda t)$ and raise the
result to the power $m_o$.  As $m_o$ is in the order of $P_1$ this can
be done in $\tilde{O}(\log(P_1) \log(q) P_1 P_2)$.

All these operations have to be repeated $O(g^2)$ times for a total
cost of $\tilde{O}(g^2 \log(q)P_1 P_2)$.

The following lemma shows that in order to express $F(m_c)$ as a
linear combination of the basis vectors of $H^1_{MW,c}(V, \pi_*
A^\dagger_K)$ it is enough to do it for the local component at the
infinity point.

\begin{lemma}
  Let $m_c=(m_{\lambda_1}, \ldots, m_{\lambda{2g+1}},m_{\infty})$ be
  an element of $H^1_{MW,c}(V, \pi_*A^\dagger_K)$ such that for each
  $\lambda$ we have $m_\lambda=Y.f_{\lambda}$ with $f_\lambda$ a power
  series in $t-\lambda$. Write $f_\infty=\sum_{\ell=0}^\infty a_\ell
  t^\ell$. If for $\ell =0, \ldots , 2g+1$, $a_\ell=0$ then $m_c=0$.
\end{lemma}
\begin{proof}
Let $a_i$ 
denote the constant term of 
$f_{\lambda_i}$. Let $t' = t^{-1}$. By Proposition \ref{eqinhom}, the power series 
$f_\infty(t')$ satisfies an equation 
$$ \frac{\partial}{\partial t'} f_\infty + t'H(t') f_\infty = u(t')$$
where $H$ is a power series
$$u(t') =\frac{1}{2}\sum_{\ell \geq 0} \sum _{i=1,\ldots, 2g+1} a_i \lambda_i^\ell t'^{\ell+1}.$$
Hence, if the first $2g+2$ coefficients of $f_\infty$ are zero then we have
$$\sum _{i=1,\ldots,2g+1} a_i \lambda_i^\ell =0$$
for $l = 1\ldots 2g+1$. Now since the $\lambda_i$ are distinct and since
the matrix
$$M=\left(
\begin{array}{cccc}
\lambda_1 & \lambda_1^2 & ... & \lambda_1^{2g+1}\\
\lambda_2 & \lambda_2^2 & ... & \lambda_2^{2g+1}\\
... & ... & ... & ... \\
\lambda_{2g+1} & \lambda_{2g+1}^2 & ... & \lambda_{2g+1}^{2g+1} \\
\end{array} \right)^\vee$$
is the transpose of a minor of a Vandermonde matrix all the $a_\lambda$ are 
zero.
\end{proof} 

Using the preceding lemma, the decomposition of $F(m_c)$ in term of
the basis vectors costs $\tilde{O}(g^3 \log(q) P_2 )$ using the
algorithm of Gauss.

In all the running time of Step 2 in $\tilde{O}(g^2 \log(q) P_1 P_2)$
and the memory consumption is $O(g \log(q)P_1 P_2+g^2 \log(q) P_2)$.

\subsubsection{Step 3: Norm computation}
Compute 
$$\tilde{M}_\Sigma = \tilde{M}_F \tilde{M}_F^\sigma \ldots \tilde{M}_F^{\sigma^n},$$

using the divide and conquer approach presented in~\cite{Kedlaya2001}.
This requires
\begin{itemize}
\item  $O(\log(n))$ multiplications of $2g \times 2g$ matrices
each one of which costs $\tilde{O}(g^3 \log(q) P_2)$ time;
\item $\tilde{O}(g^2)$ application of the Frobenius morphism at the
  expense of $\tilde{O}(\log(q) P_2)$ time
  (\cite{MR2162716}).
\end{itemize}
The overall time and memory consumption are bounded respectively  by $\tilde{O}( g^3 \log(q)
P_2)$ and $O(g^2 \log(q) P_2)$.
\begin{proposition}
Let $C_k$ be an hyperelliptic curve of genus $g$ over the finite field
$k$ with cardinality $q$. We suppose that the ramification points of
$C_k$ are rational. There exists an algorithm to compute the action of
the Frobenius morphism on $H^1_{MW,c}(C_k/K)$ with analytic precision
$P_1$ and $p$-adic precision $P_2$ with time complexity $\tilde{O}(g^2
\log(q) P_1
P_2)+\tilde{O}(g^3 \log(q) P_2)$ and memory complexity $O(g \log(q)P_1
P_2+g^2 \log(q) P_2)$.
\end{proposition}

\section{Overall complexity analysis}\label{secconl}
In this paragraph, we gather the results of Section \ref{complex1} and
Section \ref{complex2} in order to give time and memory complexity
bounds for the computation of the number of rational points of an
hyperelliptic curve defined over a finite field of characteristic $p$
and cardinality $q=p^n$ using our algorithm.

First, we have to assess the analytic $P_1$ and $p$-adic precision
$P_2$ of the computations. By the Riemann hypothesis for curves, we
know that it is enough to compute the coefficients of the matrix
$\tilde{M}_F$ with precision $g/2.n+(2g+1)\log_p(2)$.

Next, by Theorem \ref{th2}, we can take $P_1=O(P_2)$ and we get the
\begin{theorem}
  Let $C_k$ be an hyperelliptic curve of genus $g$ over the finite
  field $k$. Let $n$ be the absolute degree of $k$.  We suppose that
  the ramification points of $C_k$ are rational. There exists an
  algorithm to compute the characteristic polynomial of the Frobenius
  morphism acting on $H^1_{MW,c}(C_k/K)$ with $\tilde{O}(g^4 n^3)$
  time complexity and $O(g^3 n^3)$ memory complexity.
\end{theorem}

The algorithm described in the preceding sections have been
implemented in magma~\cite{algo}. Our implementation is only aimed at
showing the correctness of our algorithm.

\section{Conclusion}
In this paper we have described an algorithm to count the number of
rational points of an hyperelliptic curve over a finite field of odd
characteristic using Monsky-Washnitzer cohomology with compact
support.  The worst case complexity of our algorithm is quasi-cubic in
the absolute degree of the base field.  We remark that the base
computation can be easily adapted for more general curves. The reason
why we focus on the case of hyperelliptic curves in this paper is that
for more general curves the assessment of the analytic precision
necessary for the computations is more difficult. Actually, in order
to treat more general curves it is necessary to obtain an explicit
logarithmic bound for the elements of a basis of the cohomology. The
result we used in this paper guarantees such a bound provided that the
connection matrix has only simple poles and that the exponents of the
local differential equations are prepared. This last condition on the
exponents means that they are non integral or null rationals numbers
which differences are zero if they are integral. In our case, the
exponents are $0$ and $-1/2$.

\subsubsection{Acknowledgement} The authors would like to express
their deep gratitude to Bernard Le Stum for the important suggestions
and advises he gave us during the elaboration of this paper.

\newcommand{\etalchar}[1]{$^{#1}$}


\begin{thebibliography}{CGH{\etalchar{+}}96}

\bibitem[ASU75]{MR0398151}
A.~V. Aho, K.~Steiglitz, and J.~D. Ullman.
\newblock Evaluating polynomials at fixed sets of points.
\newblock {\em SIAM J. Comput.}, 4(4):533--539, 1975.

\bibitem[BCO{\etalchar{+}}07]{BostanCOSSS07}
A.~Bostan, F.~Chyzak, F.~Ollivier, B.~Salvy, {\'E}.~Schost, and A.~Sedoglavic.
\newblock Fast computation of power series solutions of systems of differential
  equations.
\newblock In {\em SODA}, pages 1012--1021, 2007.

\bibitem[BSSar]{BoSaSc08}
Alin Bostan, Bruno Salvy, and {\'E}ric Schost.
\newblock Power series composition and change of basis.
\newblock In David~J. Jeffrey, editor, {\em ISSAC'08}. ACM Press, To appear.
\newblock Proceedings of ISSAC'08, Hagenberg, Austria.

\bibitem[CFA{\etalchar{+}}06]{MR2162716}
Henri Cohen, Gerhard Frey, Roberto Avanzi, Christophe Doche, Tanja Lange, Kim
  Nguyen, and Frederik Vercauteren, editors.
\newblock {\em Handbook of elliptic and hyperelliptic curve cryptography}.
\newblock Discrete Mathematics and its Applications (Boca Raton). Chapman \&
  Hall/CRC, Boca Raton, FL, 2006.

\bibitem[CGH{\etalchar{+}}96]{MR1414285}
R.~M. Corless, G.~H. Gonnet, D.~E.~G. Hare, D.~J. Jeffrey, and D.~E. Knuth.
\newblock On the {L}ambert {$W$} function.
\newblock {\em Adv. Comput. Math.}, 5(4):329--359, 1996.

\bibitem[Cha07]{ThChat}
G.~Chatel.
\newblock Comptage de points : Application des m\'ethodes cristallines, 2007.
\newblock Available at {\small {\tt http://www.math.unipd.it/~gweltaz/ }}.

\bibitem[CL08]{algo}
G.~Chatel and D.~Lubicz.
\newblock Magma implementation of point counting algorithm using cohomology
  with compact support, 2008.
\newblock Available at {\small {\tt
  http://perso.univ-rennes1.fr/david.lubicz/programs/}}.

\bibitem[Dwo90]{MR1085482}
Bernard Dwork.
\newblock {\em Generalized hypergeometric functions}.
\newblock Oxford Mathematical Monographs. The Clarendon Press Oxford University
  Press, New York, 1990.
\newblock Oxford Science Publications.

\bibitem[Elk73]{MR0345966}
Ren{\'e}e Elkik.
\newblock Solutions d'\'equations \`a coefficients dans un anneau hens\'elien.
\newblock {\em Ann. Sci. \'Ecole Norm. Sup. (4)}, 6:553--603 (1974), 1973.

\bibitem[Elk98]{MR1486831}
Noam~D. Elkies.
\newblock Elliptic and modular curves over finite fields and related
  computational issues.
\newblock In {\em Computational perspectives on number theory (Chicago, IL,
  1995)}, volume~7 of {\em AMS/IP Stud. Adv. Math.}, pages 21--76. Amer. Math.
  Soc., Providence, RI, 1998.

\bibitem[{\'E}LS93]{MR1225991}
Jean-Yves {\'E}tesse and Bernard Le~Stum.
\newblock Fonctions {$L$} associ\'ees aux {$F$}-isocristaux surconvergents.
  {I}. {I}nterpr\'etation cohomologique.
\newblock {\em Math. Ann.}, 296(3):557--576, 1993.

\bibitem[Gau00]{Gaudry}
Pierrick Gaudry.
\newblock {\em Algorithmique des courbes hyperelliptiques et applications {\`a}
  la cryptologie}.
\newblock PhD thesis, {\'E}cole Polytechnique, 2000.

\bibitem[Gau02a]{Gaudrynt}
Pierrick Gaudry.
\newblock Cardinality of a genus 2 hyperelliptic curve over {GF}$(5\cdot
  10^{24}+41)$.
\newblock Email at the Number Theory List, September 2002.

\bibitem[Gau02b]{Gaudry2002}
Pierrick Gaudry.
\newblock A comparison and a combination of {SST} and {AGM} algorithms for
  counting points of elliptic curves in characteristic 2.
\newblock In {\em Advances in cryptology---ASIACRYPT 2002}, Lecture Notes in
  Comput. Sci. Springer, Berlin, December 2002.

\bibitem[Kat73]{MR0498577}
Nicholas Katz.
\newblock Travaux de {D}work.
\newblock In {\em S\'eminaire Bourbaki, 24\`eme ann\'ee (1971/1972), Exp. No.
  409}, pages 167--200. Lecture Notes in Math., Vol. 317. Springer, Berlin,
  1973.

\bibitem[Ked01]{Kedlaya2001}
K.S. Kedlaya.
\newblock Counting points on hyperelliptic curves using {M}onsky {W}ashnitzer
  cohomology.
\newblock {\em Journal of the Ramanujan Mathematical Society}, 16:323--328,
  2001.

\bibitem[Lau04]{MR2044050}
Alan G.~B. Lauder.
\newblock Deformation theory and the computation of zeta functions.
\newblock {\em Proc. London Math. Soc. (3)}, 88(3):565--602, 2004.

\bibitem[Lau06]{MR2261044}
Alan G.~B. Lauder.
\newblock A recursive method for computing zeta functions of varieties.
\newblock {\em LMS J. Comput. Math.}, 9:222--269 (electronic), 2006.

\bibitem[LL03]{LL03}
R.~Lercier and D.~Lubicz.
\newblock {C}ounting {P}oints on {E}lliptic {C}urves over {F}inite {F}ields of
  {S}mall {C}haracteristic in {Q}uasi {Q}uadratic {T}ime.
\newblock In Eli Biham, editor, {\em Advances in {C}ryptology---EUROCRYPT
  '2003}, Lecture Notes in Computer Science. Springer-Verlag, May 2003.

\bibitem[LL06]{MR2293798}
Reynald Lercier and David Lubicz.
\newblock A quasi quadratic time algorithm for hyperelliptic curve point
  counting.
\newblock {\em Ramanujan J.}, 12(3):399--423, 2006.

\bibitem[LS95]{MR1354350}
Bernard Le~Stum.
\newblock Filtration par le poids sur la cohomologie de de {R}ham d'une courbe
  projective non singuli\`ere sur un corps ultram\'etrique complet.
\newblock {\em Rend. Sem. Mat. Univ. Padova}, 93:43--85, 1995.

\bibitem[LW02]{MR1916921}
Alan G.~B. Lauder and Daqing Wan.
\newblock Computing zeta functions of {A}rtin-{S}chreier curves over finite
  fields.
\newblock {\em LMS J. Comput. Math.}, 5:34--55 (electronic), 2002.

\bibitem[LW04]{MR2067435}
Alan G.~B. Lauder and Daqing Wan.
\newblock Computing zeta functions of {A}rtin-{S}chreier curves over finite
  fields. {II}.
\newblock {\em J. Complexity}, 20(2-3):331--349, 2004.

\bibitem[Mes01]{Mestre4}
Jean-Fran{\c{c}}ois Mestre.
\newblock Lettre {\`a} {G}audry et {H}arley, 2001.
\newblock Available at {\tt http://www.math.jus\-sieu.fr/\homedir{}mestre}.

\bibitem[Mes02]{Mestre5}
Jean-Fran{\c{c}}ois Mestre.
\newblock Notes of a talk given at the seminar of cryptography of {R}ennes,
  2002.
\newblock Available at {\tt
  http://www.math.univ-rennes1.fr/crypto/2001-02/mestre.ps}.

\bibitem[Pap94]{MR1251285}
Christos~H. Papadimitriou.
\newblock {\em Computational complexity}.
\newblock Addison-Wesley Publishing Company, Reading, MA, 1994.

\bibitem[Pil90]{Pila1990}
J.~Pila.
\newblock Frobenius maps of abelian varieties and finding roots of unity in
  finite fields.
\newblock {\em Math. Comp.}, 55(192):745--763, 1990.

\bibitem[PS73]{MR0314238}
Michael~S. Paterson and Larry~J. Stockmeyer.
\newblock On the number of nonscalar multiplications necessary to evaluate
  polynomials.
\newblock {\em SIAM J. Comput.}, 2:60--66, 1973.

\bibitem[Sat00]{MR2001j:11049}
Takakazu Satoh.
\newblock The canonical lift of an ordinary elliptic curve over a finite field
  and its point counting.
\newblock {\em J. Ramanujan Math. Soc.}, 15(4):247--270, 2000.

\bibitem[Sch98]{Schoof1}
R.~Schoof.
\newblock Counting points on elliptic curves over finite fields.
\newblock {\em J. Th\'eorie des nombres de Bordeaux}, 7:483--494, 1998.

\bibitem[SST03]{SaSkTa2001}
T.~Satoh, B.~Skjernaa, and Y.~Taguchi.
\newblock Fast computation of canonical lifts of elliptic curves and its
  application to point counting.
\newblock {\em Finite Fields and Their Applications}, 9(1):89--101, 2003.

\bibitem[Tsu99]{MR1728610}
Nobuo Tsuzuki.
\newblock On the {G}ysin isomorphism of rigid cohomology.
\newblock {\em Hiroshima Math. J.}, 29(3):479--527, 1999.

\bibitem[vzGG03]{MR2001757}
Joachim von~zur Gathen and J{\"u}rgen Gerhard.
\newblock {\em Modern computer algebra}.
\newblock Cambridge University Press, Cambridge, second edition, 2003.

\end{thebibliography}
\end{document}